\documentclass[english,times,doublespace]{article}
\usepackage[T1]{fontenc}
\usepackage[latin9]{inputenc}
\usepackage{float}
\usepackage{amsmath}
\usepackage{amsthm}
\usepackage{amssymb}
\usepackage{esint}

\makeatletter

\providecommand{\tabularnewline}{\\}
\floatstyle{ruled}
\newfloat{algorithm}{tbp}{loa}
\providecommand{\algorithmname}{Algorithm}
\floatname{algorithm}{\protect\algorithmname}

\newcommand{\lyxaddress}[1]{
\par {\raggedright #1
\vspace{1.4em}
\noindent\par}
}
\theoremstyle{plain}
\newtheorem{thm}{\protect\theoremname}
  \theoremstyle{definition}
  \newtheorem{defn}[thm]{\protect\definitionname}
  \theoremstyle{plain}
  \newtheorem{lem}[thm]{\protect\lemmaname}
  \theoremstyle{plain}
  \newtheorem{prop}[thm]{\protect\propositionname}

\renewcommand{\vec}[1]{\mbox{\boldmath$#1$}}

\usepackage{xfrac} 
\usepackage{algorithmic}

\makeatother

\usepackage{babel}
  \providecommand{\definitionname}{Definition}
  \providecommand{\lemmaname}{Lemma}
  \providecommand{\propositionname}{Proposition}
\providecommand{\theoremname}{Theorem}

\begin{document}

\title{Overcoming Element Quality Dependence of Finite Elements with Adaptive
Extended Stencil FEM (AES-FEM)}

\author{Rebecca Conley, Tristan J. Delaney and Xiangmin Jiao{*}}

\maketitle

\lyxaddress{Department of Applied Mathematics \& Statistics, Stony Brook University,
Stony Brook, NY 11794\\
{*}xiangmin.jiao@stonybrook.edu}
\begin{abstract}
The finite element methods (FEM) are important techniques in engineering
for solving partial differential equations, but they depend heavily
on element shape quality for stability and good performance. In this
paper, we introduce the \emph{Adaptive Extended Stencil Finite Element
Method} (AES-FEM) as a means for overcoming this dependence on element
shape quality. Our method replaces the traditional basis functions
with a set of \emph{generalized Lagrange polynomial (GLP) basis functions},
which we construct using local weighted least-squares approximations.
The method preserves the theoretical framework of FEM, and allows
imposing essential boundary conditions and integrating the stiffness
matrix in the same way as the classical FEM. In addition, AES-FEM
can use higher-degree polynomial basis functions than the classical
FEM, while virtually preserving the sparsity pattern of the stiffness
matrix. We describe the formulation and implementation of AES-FEM,
and analyze its consistency and stability. We present numerical experiments
in both 2D and 3D for the Poisson equation and a time-independent
convection-diffusion equation. The numerical results demonstrate that
AES-FEM is more accurate than linear FEM, is also more efficient than
linear FEM in terms of error versus runtime, and enables much better
stability and faster convergence of iterative solvers than linear
FEM over poor-quality meshes.
\end{abstract}

\section{Introduction}

The finite element methods (FEM) are arguably one of the most important
numerical tools for solving partial differential equations (PDE) over
complex domains in engineering. They account for an overwhelming majority
of the commercial and research code for modeling and simulations,
and there is a vast amount of theoretical work to provide a rigorous
foundation; see e.g., \cite{Zienkiewicz2013FEM}.

Despite their apparent success in many applications, classical finite
element methods have a very fundamental limitation: \emph{they are
dependent on element shape quality. }\textit{\emph{This is especially
true for elliptic and parabolic problems, for which the resulting
linear system is often ill-conditioned if a mesh contains a few ``bad''
elements. This can lead to very slow convergence of iterative solvers
and sometimes even a loss of accuracy. Because of this, researchers
and users of FEM often spend a tremendous amount of time and computing
power to generate and maintain meshes, trying to fix that one last
bad element. }}This has spurred much successful research in meshing,
such as Delaunay triangulation \cite{shewchuk2002delaunay,si2015tetgen},
advancing front \cite{Lo1985}, octree-based methods \cite{Shep1991},
etc. However, the meshing problem continues to become more and more
challenging as applications become more and more sophisticated and
demanding, and also with the increased use of higher-order methods,
such as spectral element methods \cite{Patera1984}, discontinuous
Galerkin methods \cite{Arnold2002,Cockburn2000}, etc.

The FEM community has long considered this dependency on element quality
as a critical issue, and the community has been actively searching
for alternative methods to mitigate the issue for decades. Examples
of such alternative methods include the diffuse element or element-free
Galerkin methods \cite{Bely1994,nayroles1992generalizing}, least-squares
FEM \cite{bochev1998finite}, generalized or meshless finite different
methods \cite{Jensen1972,Perrone1975,Liszka1980,Benito2001,Milewski2012},
generalized or extended FEM \cite{Belytschko2009,Fries2010XFEM},
and partition-of-unity FEM \cite{melenk1996partition}. To reduce
the dependency on mesh quality, these methods avoid the use of the
piecewise-polynomial Lagrange basis functions found in the classical
FEM. However, they also lose some advantages of the classical FEM.
In particular, for the generalized finite different methods, the strong
form instead of the weak form of the PDE must be used. The partition-of-unity
FEM and other similar generalizations often incur complexities in
terms of imposing essential boundary conditions and/or integrating
the stiffness matrix \cite{melenk1996partition}. Therefore, it remains
an open problem to develop a numerical method that overcomes the element-quality
dependence, while preserving the theoretical framework of FEM, without
complicating the imposition of boundary conditions and numerical integration.

This paper introduces a new method, called the \emph{Adaptive Extended
Stencil Finite Element Method} (\emph{AES-FEM}) (pronounced as ace-F-E-M),
to address this open problem. Similar to some of the aforementioned
methods, the AES-FEM replaces the piecewise-polynomial Lagrange basis
functions in the classical FEM with alternative basis functions. Different
from those methods, our basis functions are partition-of-unity polynomial
basis functions, constructed based on local weighted least squares
approximations, over an adaptively selected stencil to ensure stability.
We refer to these basis functions as \emph{generalized Lagrange polynomial}
(\emph{GLP}) basis functions. Another difference of AES-FEM from most
other generalizations of FEM is that AES-FEM preserves the traditional
finite element shape functions as the weight functions (a.k.a. test
functions) in the weak form, to preserve the compact support of integration
and the weak form after integration by parts. This combination of
the basis and weight functions enables AES-FEM to overcome the element-quality
dependence, while preserving the theoretical framework of FEM, without
any complication in imposing essential boundary conditions or integrating
the stiffness matrix. In addition, the resulting stiffness matrix
of AES-FEM has virtually the same sparsity pattern as that of the
classical FEM, while allowing the use of higher-degree polynomials
and hence significantly improved accuracy.

As a general method, AES-FEM allows polynomial basis functions of
arbitrary degrees. In this paper, we focus on the use of quadratic
polynomial basis functions, for which the stiffness matrix has virtually
the same sparsity pattern as the classical FEM with linear basis functions.
However, AES-FEM is based on the more general weighted-residual formulation
instead of the Galerkin formulation, and hence the resulting system
is nonsymmetric, which is more expensive to solve than for symmetric
matrices. In addition, it is more expensive to construct the basis
functions of AES-FEM than to use the standard basis functions in FEM.
Therefore, AES-FEM is conceivably less efficient than FEM for a given
mesh. However, as we will demonstrate in our experimental results,
AES-FEM is significantly more accurate than FEM on a given mesh due
to its use of higher-degree basis functions, and it is also more efficient
than FEM in terms of error versus runtime. Most impotently, AES-FEM
enables better stability and faster convergence of iterative solvers
than FEM over poor-quality meshes.

The reminder of the paper is organized as follows. In Section~\ref{sec:Background-and-Related},
we present some background information and recent developments of
related methods. In Section~\ref{sec:Formulation-of-AES-FEM}, we
formulate FEM from a weighted residual perspective, describe the construction
of generalized Lagrange polynomial basis functions based on weighted
least squares approximations, and then introduce AES-FEM. In Section~\ref{sec:Analysis-of-AES-FEM},
we analyze the consistency and stability of AES-FEM. In Section~\ref{sec:Implementation},
we discuss some implementation details including the utilized data
structure and the applicable algorithms. In Section~\ref{sec:Results},
we present the results of some numerical experiments with our approach.
Finally, Section~\ref{sec:Conclusions-and-Future} concludes the
paper with a discussion.

\section{Background and Related Work\label{sec:Background-and-Related}}

In this section, we review some background information and some methods
closely related to our proposed approach, including the diffuse element
method and element free Galerkin method, some modern variants or generalizations
of FEM, as well as the generalized finite different methods.

\subsection{Diffuse Element Method and Element Free Galerkin Method}

Various alternatives of finite element methods have been proposed
in the literature to mitigate the mesh-quality dependence. One of
the examples is the \emph{diffuse element method} (DEM) \cite{nayroles1992generalizing},
proposed by Nayroles, Touzot, and Villon in 1992. Similar to AES-FEM,
DEM constructs local approximations to an unknown function based on
a local weighted least squares fitting at each node. However, unlike
AES-FEM, the DEM is based on the Galerkin formulation, which requires
the shape functions to have a compact support for efficiency. To this
end, DEM relies on a weight function that vanishes at a certain distance
from a node, in a manner similar to the moving least squares fittings
\cite{lancaster1981surfaces}. The accuracy and efficiency of numerical
integration in DEM depends on the particular choice of the weight
function. In contrast, based on a weighted-residual formulation, AES-FEM
enforces the compact support of the weak form with the weight functions,
so the shape function does not need to have a compact support. 

Another approach, which is closely related to DEM, is the \emph{element-free
Galerkin method} (EFGM) \cite{Bely1994}, proposed by Belytschko,
Lu, and Gu in 1994. As a Galerkin method, EFGM also requires a compact
support of its shape functions, which serve as both the trial functions
and weight functions. Similar to DEM, EFGM constructs the shape functions
based on moving least squares, for which the weight function plays
a critical role in terms of accuracy and efficiency. However, depending
on the weight functions, the shape functions in EFGM may not be polynomials.
It requires special quadrature rules with more quadrature points than
those of standard FEM \cite{Bely1994}, and it also requires evaluating
a moving least squares fitting at each quadrature point. In addition,
EFGM requires the use of Lagrange multipliers for essential boundary
conditions. In contrast, AES-FEM can utilize the same treatments of
boundary conditions and numerical integration as the standard FEM.

\subsection{Other Variants and Generalizations of FEM}

Besides DEM and EFGM, various other generalizations of FEM have been
developed in recent years. Some of the most notable ones include the
generalized finite element method (GFEM) and extended finite element
method (XFEM) \cite{Belytschko2009,duarte1996h,Duarte2000GFEM,srinivasan2008generalized,Fries2010XFEM}.
These methods also alleviate the dependence on the mesh, by introducing
enrichment functions to replace the standard FEM basis functions in
regions with discontinuous solutions, such as along cracks. These
enrichment functions in general are not polynomials, so special quadrature
rules may be needed for efficiency. Away from the discontinuities,
GFEM and XFEM rely on the standard FEM discretizations, so good mesh
quality is still required in general.

The GFEM and XFEM may be viewed as special cases of the partition
of unity method (PUM) \cite{melenk1996partition}, which is a general
framework for constructing alternative basis functions. As noted in
\cite{melenk1996partition}, the main outstanding questions of PUM
include the choice of the basis functions, the imposition of essential
boundary conditions, and efficient numerical integration. The AES-FEM
proposes a set of generalized Lagrange polynomial basis functions
that also satisfy the partition of unity. Therefore, AES-FEM effectively
addresses these open problems in PUM.

Besides the aforementioned generalizations of FEM, it is also worth
noting the least-squares finite element method (LSFEM) \cite{bochev1998finite}.
LSFEM uses the concept of least squares globally to minimize a global
error. In contrast, AES-FEM uses least squares in a local sense for
constructing basis functions. To use LSFEM, any higher order PDE must
be decomposed into a system of first order PDEs first, so it does
not preserve the framework of the standard FEM. 

Finally, we note the recent development of isogeometric analysis (IGA)
\cite{hughes2005isogeometric}, which uses NURBS (Non-Uniform Rational
B-Splines) or T-splines as basis functions instead of the standard
FEM basis functions. These methods can deliver high accuracy over
very coarse meshes and can be advantageous for problems that can benefit
from high-degree continuity, such as thin-shell modeling. However,
IGA does not alleviate the dependency on mesh quality, since NURBS
in effect impose stronger requirement on mesh quality than the standard
FEM.

\subsection{Generalized Finite Difference and Weighted Least Squares}

The finite difference methods and finite element methods are closely
related to each other. On structured meshes, the equivalence of these
methods can be established in some special cases. While finite element
methods were developed to support unstructured meshes from the beginning,
the finite difference methods can also be generalized to unstructured
meshes. These generalizations are often collectively referred to as
\emph{generalized finite difference }(or \emph{GFD}) methods. The
earlier GFD methods were based on polynomial interpolation; see e.g.,
\cite{Forsythe1960Finite}, \cite{Jensen1972}, and \cite{DKL84CSCR}.
These methods construct a local multivariate polynomial interpolation
by requiring the number of points in the stencil to be equal to the
number of coefficients in the interpolant. However, due to the irregular
distribution of points, the resulting Vandermonde matrices are often
singular or ill-conditioned. 

More general than an interpolant are the least squares or weighted
least squares approximations, which allow more points in the stencil
than the number of coefficients of a polynomial. Some earlier examples
of least-squares approximations include \cite{Perrone1975} and \cite{Liszka1980},
which attempted to improve the conditioning of interpolation-based
GFD. More recently, the least-squares-based GFD have been analyzed
more systematically by Benito et al. \cite{Benito2001,Gavete2003},
and it been successfully applied to the solutions of parabolic and
hyperbolic PDEs \cite{Benito2001,Benito2007,Gavete2003} and of advection-diffusion
equation \cite{Prieto2011}. It has also been utilized in the weak-form
of the Poisson equation, under the name \emph{meshless finite difference
method} \cite{Milewski2012,Milewski2013}. In these methods, a weighting
scheme is often used to alter the norm that is being minimized, and
hence the name \emph{weighted least squares} (WLS). These weights
serve a different role from those in the moving least squares \cite{lancaster1981surfaces},
DEM, and EFGM. They do not need to have a compact support, and do
not even to be defined by a continuous function in general.

Even though the least-squares approximations tend to be better conditioned
than their interpolation-based counterparts, ill-conditioning may
still occur for a given set of points. To overcome the issue, Jiao
et al. utilized adaptive stencils coupled with column and row scaling,
QR factorization with column pivoting, and a condition number estimator,
which effectively guarantee the conditioning and hence the stability
of approximations based on WLS \cite{Jiao2008,Wang2009}. In AES-FEM,
we extend this previous work to construct the generalized Lagrange
polynomial basis functions, and then utilize these basis functions
in the weak form of the finite element methods.

\section{Formulation of AES-FEM \label{sec:Formulation-of-AES-FEM}}

The main idea of AES-FEM is the use of an alternative set of basis
functions. For the basis functions, we propose the use of a set of
generalized Lagrange polynomial basis functions (GLPBF) computed using
a weighted least squares formulation. We use the standard FEM (hat)
basis functions for the weight functions (a.k.a. test functions).
In this section, we describe the weighted residual formulation of
FEM, define generalized Lagrange polynomial basis functions based
on weighted least squares, and then describe AES-FEM.

\subsection{Weighted Residual Formulation of FEM }

We will approach the formulation of the finite element method through
the lens of the weighted residual method for solving PDEs. The main
idea of the formulation is to consider the unknown solution as a linear
combination of basis (trial) functions and then select the coefficients
such that the residual is orthogonal to a set of weight (test) functions.
Depending on the choice of the weight functions, one will derive different
numerical methods, such as the collocation method, the least squares
method, and the Galerkin method. Details about weighted residual and
FEM can be found in \cite{brenner2008mathematical,ciarlet2002finite,finlayson1973method}.
In the following, we give a brief overview of weighted residuals for
completeness.

Consider a linear differential operator $\mathcal{L}$ defined on
a bounded, simply-connected domain $\Omega$, with outward unit normal
vector $\vec{n}$. Denote the boundary of $\Omega$ as $\Gamma=\Gamma_{D}\cup\Gamma_{N}$,
where $\Gamma_{D}$ and $\Gamma_{N}$ are disjoint sets on which Dirichlet
and Neumann boundary conditions are specified, respectively. We want
to find a function $u$ such that
\begin{equation}
\mathcal{L}u=f\label{eq:strong_form}
\end{equation}
subject to the boundary conditions
\begin{equation}
u=g\mbox{ on }\Gamma_{D}\quad\mbox{ and }\quad\frac{\partial u}{\partial\vec{n}}=h\text{ on }\Gamma_{N}.\label{eq:boundcond}
\end{equation}
Eq.~(\ref{eq:strong_form}) is the strong form of the PDE. In the
weighted residual formulation, we use the weak form based on a set
of weight functions $\Psi=\left\{ \psi_{1},\dots,\psi_{n}\right\} $,
by requiring the residual $\mathcal{L}u-f$ to be orthogonal to $\psi_{i}$,
i.e., 
\begin{equation}
\int_{\Omega}\psi_{i}\left(\mathcal{L}u-f\right)\,dV=0.\label{eq:weak_form}
\end{equation}
To approximate $u$, let $\Phi=\left\{ \phi_{1},\dots,\phi_{n}\right\} $
be a set of basis functions, and we define an approximation
\begin{equation}
u\approx\sum_{j=1}^{n}u_{j}\phi_{j}.\label{eq:approximation}
\end{equation}
Substituting (\ref{eq:approximation}) into the weak form (\ref{eq:weak_form})
and rearranging the equations, we then obtain
\begin{equation}
\sum_{j=1}^{n}u_{j}\int_{\Omega}\psi_{i}\left(\mathcal{L}\phi_{j}\right)\,dV=\int_{\Omega}\psi_{i}f\,dV.\label{eq:approx_weakform}
\end{equation}
At this point for simplicity, let us consider the Poisson equation
with Dirichlet boundary conditions, for which the weak form is given
by 
\begin{equation}
\int_{\Omega}\psi_{i}\nabla^{2}u\,dV=\int_{\Omega}\psi_{i}f\,dV.\label{eq:weakform_Poisson}
\end{equation}
Substituting (\ref{eq:approximation}) into (\ref{eq:weakform_Poisson}),
we obtain
\begin{equation}
\sum_{j=1}^{n}u_{j}\int_{\Omega}\psi_{i}\nabla^{2}\phi_{j}\ dV=\int_{\Omega}\psi_{i}f\ dV.\label{eq:weight_residual_Poisson}
\end{equation}
The finite element method uses integration by parts to reduce the
order of derivatives required by (\ref{eq:weight_residual_Poisson}).
If $\psi_{i}$ has weak derivatives and satisfies the condition $\psi_{i}|_{\Gamma_{D}}=0$,
then after integrating by parts and imposing the boundary conditions,
we arrive at
\begin{equation}
-\sum_{j=1}^{n}u_{j}\int_{\Omega}\nabla\psi_{i}\cdot\nabla\phi_{j}\ dV=\int_{\Omega}\psi_{i}f\ dV.\label{eq:approx_weakform_Poisson}
\end{equation}
Taking (\ref{eq:approx_weakform_Poisson}) over the $n$ weight functions,
we obtain the linear system
\begin{equation}
\vec{K}\vec{u}=\vec{g},
\end{equation}
where $\vec{K}$ is the stiffness matrix and $\vec{g}$ is the load
vector, with 
\begin{equation}
k_{ij}=-\int_{\Omega}\nabla\psi_{i}\cdot\nabla\phi_{j}\ dV\qquad\text{and }\qquad g_{i}=\int_{\Omega}\psi_{i}f\ dV.\label{eq:stiffness_entries}
\end{equation}
If the weight functions are chosen to be the same as the basis functions,
then we will arrive at the Galerkin method. In this paper, we introduce
a new set of basis functions based on weighted least squares and we
use the standard linear FEM ``hat functions'' as the weight functions.

\subsection{Weighted Least Squares Approximations \label{sub:GFD_=000026_WLS}}

In this subsection, we review numerical differentiation based on weighted
least squares approximations, as described in \cite{Jiao2008,Wang2009}.
Similar to interpolation-based approximations, this method is based
on Taylor series expansion. Let us take 2D as an example, and suppose
$f(\vec{u})$ is a bivariate function with at least $d+1$ continuous
derivatives in some neighborhood of $\vec{u}_{0}=(0,0)$. Denote $c_{jk}=\frac{\partial^{j+k}}{\partial u^{j}\partial v^{k}}f(\vec{u}_{0})$.
Then for any $\vec{u}$ in the neighborhood, $f$ may be approximated
to the $(d+1)$st order accuracy about the origin $\vec{u}_{0}$ as
\begin{equation}
f(\vec{u})=\underbrace{\sum_{p=0}^{d}\sum_{j,k\geq0}^{j+k=p}c_{jk}\frac{u^{j}v^{k}}{j!k!}}_{\text{Taylor Polynomial}}+\underbrace{\mathcal{O}(\left\Vert \vec{u}\right\Vert ^{d+1})}_{\text{remainder}}.\label{eq:Taylor_GFD_expansion}
\end{equation}
Analogous formulae exist in 1D and 3D. The derivatives of the Taylor
polynomial are the same as those of $f$ at $\vec{u}_{0}$ up to degree
$d$. Therefore, once we have calculated the coefficients for the
Taylor polynomial, finding the derivatives of $f$ at $\vec{u}_{0}$
is trivial. We proceed with calculating the coefficients as follows.

For any point $\vec{u}_{0}$, we select a stencil of $m$ nodes from
the neighborhood around $\vec{u}_{0}$. Stencil selection is described
further in Section~\ref{sec:Implementation}. We do a local parameterization
of the neighborhood so that $\vec{u}_{0}$ is located at the origin
$(0,0)$ and the coordinates of the other points are given relative
to $\vec{u}_{0}$. Then substituting these points into (\ref{eq:Taylor_GFD_expansion}),
we obtain a set of approximate equations 
\begin{equation}
\sum_{p=0}^{d}\sum_{j,k\geq0}^{j+k=p}c_{jk}\frac{u_{i}^{j}v_{i}^{k}}{j!k!}\approx f_{i},
\end{equation}
where $f_{i}=f\left(\vec{u}_{i}\right)$ and the $c_{jk}$ denote
the unknowns, resulting in an $m\times n$ system. There are $n=(d+1)(d+2)/2$
unknowns in 2D and $n=(d+1)(d+2)(d+3)/6$ unknowns in 3D. Let $\vec{V}$
denote the generalized Vandermonde matrix, $\vec{c}$ denote the vector
of unknowns (i.e., the $c_{jk}$) and $\vec{f}$ denote the vector
of function values. Then we arrive at the rectangular system 
\begin{equation}
\vec{V}\vec{c}\approx\vec{f}.\label{eq:Vc_equal_f}
\end{equation}

Let us now introduce some notation to allow us to write the Taylor
series in matrix notation before we proceed with our discussion of
solving (\ref{eq:Vc_equal_f}). Let $\vec{\mathcal{P}}_{k}^{(d)}(\vec{x})$
denote the set of all $k$-dimensional monomials of degree $d$ and
lower, stored in ascending order as a column vector. If no ambiguities
will arise, we will use $\vec{\mathcal{P}}$ in place of $\vec{\mathcal{P}}_{k}^{(d)}(\vec{x})$.
For example, for second degree in 2D we have
\begin{equation}
\vec{\mathcal{P}}_{2}^{(2)}\left(\vec{x}\right)=\left[1\quad x\quad y\quad x^{2}\quad xy\quad y^{2}\right]^{T}.
\end{equation}
Let $\vec{D}$ be a diagonal matrix consisting of the fractional factorial
part of the coefficients, i.e. $\frac{1}{j!k!}$ in (\ref{eq:Taylor_GFD_expansion}).
For example, for second degree in 2D we have
\begin{equation}
\vec{D}=\text{diag}\left(1,\ 1,\ 1,\ \frac{1}{2},\ 1,\ \frac{1}{2}\right).
\end{equation}
 Then we may write the Taylor series as 
\begin{equation}
f(\vec{x})=\vec{c}^{T}\vec{D}\vec{\mathcal{P}}\left(\vec{x}\right).
\end{equation}

To solve (\ref{eq:Vc_equal_f}), we use a weighted linear least squares
formulation \cite{Golub13MC}, that is, we will minimize a weighted
norm (or semi-norm)
\begin{equation}
\min_{\vec{c}}\left\Vert \vec{V}\vec{c}-\vec{f}\right\Vert _{\vec{W}}\equiv\min_{\vec{c}}\left\Vert \vec{W}\left(\vec{V}\vec{c}-\vec{f}\right)\right\Vert _{2},
\end{equation}
where $\vec{W}$ is an $m\times m$ diagonal weighting matrix. The
entries of $\vec{W}$ assign weights to the rows of matrix $\vec{V}$.
Specifically, if we denote the diagonal entries of $\vec{W}$ as $w_{i}$,
then row $i$ is assigned weight $w_{i}$. These weights can be used
to prioritize the points in the system: we assign heavier weights
to the nodes that are closer to the center point. By setting a weight
to zero (or very close to zero), we may also filter out outliers or
other undesirable points. Note that for a given node, the weighting
matrix $\vec{W}$ is constant.

If $\vec{f}$ is in the column space of $\vec{V}$, then the solution
of the linear system is not affected by a nonsingular weighting matrix.
However, if $\vec{f}$ is not in the column space, which is often
the case, then different weighting schemes can lead to different solutions.
Choosing a good weighting matrix is application specific. For quadratic
approximations, we compute the weights as follows. Let $h$ denote
the maximum radius of the neighborhood, that is 
\begin{equation}
h=\max_{1\leq i\leq m}\left\{ \left\Vert \vec{u}_{i}\right\Vert _{2}\right\} .
\end{equation}
Then 
\begin{equation}
w_{i}=\left(\frac{\left\Vert \vec{u}_{i}\right\Vert _{2}}{h}+\epsilon\right)^{-1},
\end{equation}
where $\epsilon$ is a small number, such as $\epsilon=0.01$, for
avoiding division by zero.

After the weighting matrix has been applied, we can denote the new
system as
\begin{equation}
\vec{M}\vec{c}\approx\tilde{\vec{f}},\quad\mbox{where}\quad\vec{M}=\vec{W}\vec{V}\mbox{ and }\tilde{\vec{f}}=\vec{W}\vec{f}.
\end{equation}
This resulting system may be rank-deficient or ill-conditioned. This
is a challenge that GFD researchers have been dealing with since the
1970s \cite{Jensen1972}. The ill-conditioning may arise from a number
of issues including poor scaling, an insufficient number of nodes
in the neighborhood, or a degenerate arrangement of points. We resolve
these issues with neighborhood selection, discussed in Section~\ref{sec:Implementation}.
We can address the scaling issue with the use of a diagonal scaling
matrix $\vec{S}$. Let $\vec{a}_{j}$ denote the $j\text{th}$ column
of an arbitrary matrix $\vec{A}$. A typical choice for the $j\text{th}$
entry of $\vec{S}$ is either $1/\left\Vert \vec{a}_{j}\right\Vert _{2}$,
which approximately minimizes the 2-norm condition number of $\vec{A}\vec{S}$
\cite{Golub13MC}, or $1/\left\Vert \vec{a}_{j}\right\Vert _{\infty}$
\cite{Cazals2005}. Using exact arithmetic, the matrix $\vec{S}$
does not affect the solution, but it can significantly improve the
conditioning and thus the accuracy in the presence of rounding errors.
After applying the scaling matrix to $\vec{W}\vec{V}$, the problem
becomes 
\begin{equation}
\min_{\vec{d}}\left\Vert \tilde{\vec{V}}\vec{d}-\tilde{\vec{f}}\right\Vert _{2},\quad\text{where }\tilde{\vec{V}}\equiv\vec{W}\vec{V}\vec{S}=\vec{M}\vec{S}\text{ and }\vec{d}\equiv\vec{S}^{-1}\vec{c}.\label{eq:WLS}
\end{equation}
Conceptually, the solution to the above problem may be reached through
the use of a pseudoinverse. We will have 
\begin{equation}
\vec{d}=\tilde{\vec{V}}^{+}\tilde{\vec{f}}\quad\text{ where }\tilde{\vec{V}}^{+}\equiv\left(\tilde{\vec{V}}^{T}\tilde{\vec{V}}\right)^{-1}\tilde{\vec{V}}^{T}.\label{eq:coefficient_pseduo_inverse}
\end{equation}
However, since the resulting system may still be rank-deficient or
ill-conditioned, we solve it using QR factorization with column pivoting,
as discussed in Subsection~\ref{sub:Algorithms}. Finally, we get
the vector of partial derivatives for the Taylor polynomial 
\begin{equation}
\vec{c}=\vec{S}\vec{d}.\label{eq:final_equ_for_c_vector}
\end{equation}

\subsection{Generalized Lagrange Polynomial Basis Functions}

We now define basis functions based on weighted least squares. Note
that the standard finite element methods use piecewise Lagrange polynomial
shape functions, which have two especially important properties: the
coefficients of the basis functions have the physical meaning of the
function values or their approximations at the nodes, and the basis
functions form a partition of unity. We refer to the two properties
as \emph{function value as coefficient} and \emph{partition of unity},
respectively. These properties are desirable in ensuring the consistency
of the method based on these basis functions and also for the ease
of imposing Dirichlet boundary conditions. However, the traditional
concept of the Lagrange basis functions is interpolatory, so they
are not applicable to least squares. We now generalize this concept,
so that it can be applicable to least-squares-based basis functions.
\begin{defn}
Given a set of degree-$d$ polynomial basis functions $\left\{ \phi_{i}\right\} $,
we say it is a set of degree-$d$ \emph{generalized Lagrange }\textit{polynomial
(GLP)}\emph{ basis functions} if:\end{defn}
\begin{enumerate}
\item $\sum_{i}f\left(x_{i}\right)\phi_{i}$ approximates a function $f$
to $\mathcal{O}\left(h^{d+1}\right)$ in a neighborhood of the stencil,
where $h$ is some characteristic length measure, and 
\item $\sum_{i}\phi_{i}=1.$
\end{enumerate}
We now define a set of GLP basis functions based on weighted least
squares. 

Given a stencil $\left\{ x_{i}\right\} $, we follow the procedure
in Subsection~\ref{sub:GFD_=000026_WLS}. When computing the $j$th
basis function $\phi_{j}$, let $\vec{f}=\vec{e}_{j}$, where $\vec{e}_{j}$
is the $j$th column of the identity matrix. Following (\ref{eq:coefficient_pseduo_inverse})
and (\ref{eq:final_equ_for_c_vector}), we have
\begin{equation}
\vec{c}=\vec{S}\tilde{\vec{V}}^{+}\vec{W}\vec{e}_{j}.
\end{equation}
Thus for the $j$th basis function, the vector $\vec{c}$ is exactly
the $j$th column of $\vec{S}\tilde{\vec{V}}^{+}\vec{W}$. We define
a set of basis functions as 
\begin{equation}
\vec{\Phi}=\left(\vec{S}\tilde{\vec{V}}^{+}\vec{W}\right)^{T}\vec{D}\vec{\mathcal{P}}.\label{eq:GFD_Definition}
\end{equation}
For a more concrete example, if we denote the elements of $\tilde{\vec{V}}^{+}$
as $a_{ij}$ we can see that the $j$th basis function for degree
2 in 2D is
\begin{equation}
\phi_{j}=w_{j}\left(a_{1j}s_{1}+a_{2j}s_{2}x+a_{3j}s_{3}y+a_{4j}s_{4}\frac{1}{2}x^{2}+a_{5j}s_{5}xy+a_{6j}s_{6}\frac{1}{2}y^{2}\right).
\end{equation}
Note that $w_{j}$ is a constant scalar, so $\phi_{j}$ is a polynomial.
The basis functions in (\ref{eq:GFD_Definition}) are an example of\emph{
}GLP basis functions. We summarize this key feature in the following
theorem.
\begin{thm}
\label{thm:GFD-GLBF}The basis functions based on weighted least squares
as defined in (\ref{eq:GFD_Definition}) are generalized Lagrange
polynomial basis functions.
\end{thm}
We shall postpone the proof of this theorem to Section~\ref{sec:Analysis-of-AES-FEM},
where we will also analyze the accuracy and stability of finite element
discretization based on these basis functions. In the following, we
will finish the description of AES-FEM.

\subsection{Description of AES-FEM \label{sub:Description-of-AES-FEM}}

Starting with the weighted residual formulation for FEM from (\ref{eq:approx_weakform_Poisson}),
we propose using GLP basis functions for the basis functions in the
weak form and using the traditional hat functions for the weight functions.
More specifically, for a given node $i$ and its associated weight
function $\psi_{i}$, a specific set of GLP basis functions $\left\{ \phi_{j}\right\} $
is constructed from a weighted stencil $\left(\vec{X}_{i},\vec{w}_{i}\right)$
of $n$ neighboring vertices centered at node $i$. This weight function
$\psi_{i}$ and its associated set of GLP basis functions $\left\{ \phi_{j}\right\} $
are used to compute the $i$th row of the stiffness matrix, as given
by (\ref{eq:stiffness_entries}). Note that, because a different set
of basis functions is associated with each weight function, the basis
functions on a given element differ row-to-row in the stiffness matrix. 

For AES-FEM, when using the 1-ring neighborhood of a vertex as the
stencil, we can use quadratic GLP basis functions, which have the
advantage of decreased dependence on element quality and improved
accuracy over standard finite element with linear basis functions,
while virtually preserving the sparsity pattern of the stiffness matrix,
as we will demonstrate in Section~\ref{sec:Results}. 

In terms of the computation of the load vector, we can use either
the FEM or AES-FEM basis functions. We refer to the AES-FEM with these
two options as ``AES-FEM 1'' and ``AES-FEM 2,'' respectively.
Additional implementation details will be given in Subsection \ref{sub:Algorithms}.

\section{Analysis of AES-FEM\label{sec:Analysis-of-AES-FEM}}

In this section, we analyze the consistency and stability of AES-FEM.
We start by explaining the applicability of Green's identity to the
GLP basis functions. We will then prove that basis functions in (\ref{eq:GFD_Definition})
are GLP basis functions and discuss the consistency and stability
of AES-FEM.

\subsection{Green's Identity and Integration by Parts \label{sub:General-Properties}}

For a given weight function $\psi$, the GLP basis functions are continuously
differentiable everywhere in within the domain of integration. Hence,
in a \emph{variational sense}, the GLP basis functions still allow
one to use Green's identities to formulate a weak form. Let $\phi_{j}$
be any GLP basis function computed from a weighted stencil, and $\psi_{i}$
be a classical FEM shape function with compact support contained within
the set $\Omega$. Therefore, for any partial derivative operation
$\partial$, it follows that 
\begin{equation}
\int_{\Omega}\left(\partial\phi_{j}\right)\psi_{i}\ \mathrm{d}x=-\int_{\Omega}\phi_{j}\left(\partial\psi_{i}\right)\ \mathrm{d}x.
\end{equation}
Because of this property, the weak-form formulation with the GLP basis
functions in AES-FEM is mathematically sound.

In addition, since the finite-element basis functions $\psi_{i}$
vanish along the boundary, after integration by parts, we have the
identities 
\begin{equation}
\int_{\Omega}\psi_{i}\nabla^{2}\phi_{j}\ dV=-\int_{\Omega}\nabla\psi_{i}\cdot\nabla\phi_{j}\ dV.
\end{equation}
Therefore, we can reduce the order of derivatives similar to the classical
FEM, without introducing additional boundary integrals to the computation.

\subsection{Generalized Lagrange Polynomial Basis Functions \label{sub:GFD-are_GL}}

We now show that the basis functions in (\ref{eq:GFD_Definition})
are indeed generalized Lagrange polynomial basis functions, which
follows from the two lemmas below.
\begin{lem}
\label{lem:GFD_coeff_equal_f}Let $\left\{ x_{i}\right\} $ be a stencil
with $m$ nodes and stencil diameter $h$. Let $\left\{ \phi_{j}\right\} $
be the complete set of basis functions of degree up to $d$ as defined
by (\ref{eq:GFD_Definition}) on this stencil. Given an arbitrary
function $f$, define the following approximation of $f$ 
\begin{equation}
f^{h}(x)\equiv\sum_{j=1}^{m}f_{j}\phi_{j}(x)=\vec{f}^{T}\vec{\Phi},
\end{equation}
where $f_{j}=f\left(x_{j}\right)$ and $\vec{f}=\left[f_{1}\ f_{2}\ \dots\ f_{m}\right]^{T}$.
If the rescaled matrix $\tilde{\vec{V}}$ has a bounded condition
number, then $f^{h}$ approximates $f$ to $\mathcal{O}\left(h^{d+1}\right)$.
In addition, given any degree-$k$ differential operator $\mathcal{D}$,
if $f$ is continuously differentiable up to degree $k$, then $\mathcal{D}f^{h}$
approximates $\mathcal{D}f$ to $\mathcal{O}\left(h^{d-k+1}\right)$.\end{lem}
\begin{proof}
First, we show that the approximation $f^{h}$ is equivalent to directly
solving a weighted least squares problem for a local polynomial fitting
of $f$. Using the method described in Subsection~\ref{sub:GFD_=000026_WLS},
we get the coefficients $\vec{c}=\vec{S}\tilde{\vec{V}}\ ^{+}\vec{W}\vec{g}$
where $\vec{g}=\left[f_{1}\ f_{2}\ \dots\ f_{m}\right]^{T}$. Thus,
the local polynomial fitting of $f$ is 
\begin{align*}
f\left(\vec{x}\right) & \approx\vec{c}^{T}\vec{D}\vec{\mathcal{P}}\\
 & =\left(\vec{S}\tilde{\vec{V}}^{+}\vec{W}\vec{g}\right)^{T}\vec{D}\vec{\mathcal{P}}\\
 & =\vec{g}^{T}\left(\vec{S}\tilde{\vec{V}}^{+}\vec{W}\right)^{T}\vec{D}\vec{\mathcal{P}}\\
 & =\vec{g}^{T}\vec{\Phi}\\
 & =f^{h}\left(\vec{x}\right).
\end{align*}
It follows from Theorem~4 in \cite{Jiao2008} that $f^{h}$ approximates
$f$ to $\mathcal{O}\left(h^{d+1}\right)$, and $\mathcal{D}f^{h}$
approximates $\mathcal{D}f$ to $\mathcal{O}\left(h^{d-k+1}\right)$
for any degree-$k$ differential operator $\mathcal{D}$.\end{proof}
\begin{lem}
The basis functions in (\ref{eq:GFD_Definition}) form a partition
of unity, i.e., $\sum_{j=1}^{m}\phi_{j}(x)=1$.\end{lem}
\begin{proof}
We will show that on a given stencil $\left\{ x_{i}\right\} $, the
GLP basis functions $\left\{ \phi_{j}\right\} $ form a partition
of unity. Let $\vec{V}$ be the generalized Vandermonde matrix for
the given stencil. For example, for second order expansion in 2D,
we have 
\begin{equation}
\vec{V}=\left[\begin{array}{ccccc}
1 & x_{1} & y_{1} & \dots & \frac{1}{2}y_{1}^{2}\\
1 & x_{2} & y_{2} & \dots & \frac{1}{2}y_{2}^{2}\\
\vdots & \vdots & \vdots & \ddots & \vdots\\
1 & x_{m} & y_{m} & \dots & \frac{1}{2}y_{m}^{2}
\end{array}\right].
\end{equation}
For a given function $f$, using the truncated Taylor series, we have
\begin{equation}
\vec{V}\vec{c}=\vec{g},
\end{equation}
where $\vec{c}$ is the vector of partial derivative values. Applying
the diagonal row weighting matrix $\vec{W}$ and the diagonal column
scaling matrix $\vec{S}$, we have
\begin{equation}
\tilde{\vec{V}}\left(\vec{S}^{-1}\vec{c}\right)=\vec{W}\vec{g}\quad\text{where}\quad\tilde{\vec{V}}\equiv\vec{W}\vec{V}\vec{S}.
\end{equation}
This is a least squares problem, and the solution for $\vec{c}$ may
be reached through the use of a pseudoinverse, 
\begin{equation}
\vec{c}=\vec{S}\tilde{\vec{V}}^{+}\vec{W}\vec{g}.
\end{equation}
For the $j$th GLP basis function, we have $\vec{g}_{j}=[0\dots1\dots0]^{T}$,
where the 1 is in the $j$th position, and hence the columns of $\vec{S}\tilde{\vec{V}}^{+}\vec{W}$
multiplied by the Taylor constants $\vec{D}$ give the coefficients
for the GLP basis functions. This implies that the entries of the
$i$th row of $\vec{S}\tilde{\vec{V}}^{+}\vec{W}$ correspond to the
coefficients of the $i$th terms in the set of basis functions.

To finish the proof, it suffices to show that the sum of the entries
in the first row of $\vec{S}\tilde{\vec{V}}^{+}\vec{W}$ is 1, and
the sum of the entries in any other row is 0. Let vector $\vec{w}$
be the diagonal entries of $\vec{W}$. Every entry of the first column
of $\vec{V}$ is equal to 1, thus the first column of $\tilde{\vec{V}}$
is then $s_{1}\vec{w}$. Denote the $i$th row of $\tilde{\vec{V}}^{+}$
as $\tilde{\vec{v}}_{(i,:)}^{+}$. The sum of the entries of the $i$th
row of $\vec{S}\tilde{\vec{V}}^{+}\vec{W}$ is $s_{1}\tilde{\vec{v}}_{(i,:)}^{+T}\vec{w}$.
Since $\tilde{\vec{V}}^{+}$ is a left inverse of $\tilde{\vec{V}}$,
we have $\tilde{\vec{V}}^{+}\tilde{\vec{V}}=\vec{I}$, and hence 
\begin{equation}
s_{1}\tilde{\vec{v}}_{(i,:)}^{+T}\vec{w}=\begin{cases}
1 & i=1\\
0 & 2\leq i\leq n
\end{cases}.
\end{equation}
Therefore, the GLP basis functions form a partition of unity.
\end{proof}
From the above lemmas, the basis functions in (\ref{eq:GFD_Definition})
satisfy both the properties of function value as coefficient and partition
of unity, and hence they are GLP basis functions, as claimed in Theorem~\ref{thm:GFD-GLBF}.

\subsection{Consistency of AES-FEM}

The accuracy of the AES-FEM depends on its consistency and stability.
We first consider the consistency of AES-FEM, in terms of its truncation
errors in approximating the weak form (\ref{eq:weak_form}) by (\ref{eq:approx_weakform})
using the GLP basis functions. In a nutshell, the consistency of the
AES-FEM follows directly from Lemma~\ref{lem:GFD_coeff_equal_f}.
For completeness, we consider a specific example of solving the Poisson
equation. The analysis for other PDEs can be derived in a similar
fashion.
\begin{thm}
Suppose $u$ is smooth and thus $\Vert\nabla u\Vert$ is bounded.
Then, when solving the Poisson equation using AES-FEM with degree-$d$
GLP basis functions in (\ref{eq:approx_weakform_Poisson}), for each
$\psi_{i}$ the weak form (\ref{eq:weakform_Poisson}) is approximated
to $\mathcal{O}(h^{d})$, where $h$ is some characteristic length
measure of the mesh. \end{thm}
\begin{proof}
Let $u$ be the exact solution on a mesh with mesh size $h$, and
let 
\[
\tilde{u}=\sum_{j=1}^{n}u_{i}\phi_{j}
\]
denote the approximation to $u$ using degree-$d$ GLP basis functions.
Because the test functions $\psi_{i}$ vanishes along boundary, the
weak form (\ref{eq:weakform_Poisson}) can be rewritten as 
\[
-\int_{\Omega}\nabla u\cdot\nabla\psi_{i}\,dV=\int_{\Omega}\psi_{i}f\,dV.
\]
When using degree-$d$ GLP basis functions, it follows from Lemma~\ref{lem:GFD_coeff_equal_f}
that 
\[
\Vert\nabla u-\nabla\tilde{u}\Vert=\mathbb{\mathcal{O}}\left(h^{d}\right)
\]
within each element. Under the assumption that $u$ is twice differentiable,
$\nabla u$ is bounded, and hence
\[
\left\vert \int_{\Omega}\left(\nabla u-\nabla\tilde{u}\right)\cdot\nabla\psi_{i}\ dV\right\vert =\mathbb{\mathcal{O}}\left(h^{d}\right).
\]
Furthermore, if $f$ is approximated by degree-$d$ GLP basis functions
as 
\[
\tilde{f}=\sum_{j=1}^{n}f_{i}\phi_{j},
\]
then we have
\[
\left\vert \int_{\Omega}\left(f-\tilde{f}\right)\psi_{i}\ dV\right\vert =\mathbb{\mathcal{O}}\left(h^{d+1}\right).
\]

\end{proof}
More specifically, when using quadratic GLP basis functions, the truncation
errors are second order in the stiffness matrix. The truncation errors
in the load vector is third order for AES-FEM 2 when $f$ is also
approximated using quadratic GLP basis functions, but it is second
order for AES-FEM 1 when $f$ is approximated using linear FEM hat
functions. Like most other PDE methods, as long as the method is stable,
the rounding errors do not dominate the truncation errors, and there
is no systematic cancelation of truncation errors, we expect the solution
to converge at the same rate as the local truncation errors, as we
will demonstrate numerically in Section~\ref{sec:Results}.

\subsection{Stability \label{sub:Stability-and-Accuracy}}

For elliptic PDEs, the stability of a method depends on the condition
number of its coefficient matrix, which can affect the performance
of iterative solvers and the accuracy of the solution. It is well
known that the traditional finite element method may be unstable for
poorly shaped meshes \cite{babuska1976angle}, and some meshless methods
may also suffer from instability when two points nearly coincide.
AES-FEM avoids these potential instability issues. 

As a concrete example, let us consider the Poisson equations with
Dirichlet boundary conditions, whose coefficient matrix is the stiffness
matrix. It is well known that the condition number of the stiffness
matrix is proportional to $h^{-2}$, where $h$ is some characteristic
length of the mesh \cite{Lui12NAPDE}. However, if the condition number
is significantly larger, then the method is said to be \emph{unstable},
which can happen due to various reasons. 

The ill-conditioning of any local stiffness matrix may lead to poor
scaling and in turn ill-conditioning of the global stiffness matrix.
This is owing to the following fact, which is given as Theorem 2.2.26
in \cite{watkins2004fundamentals}.
\begin{prop}
\label{prop:condition-number-estimate}For any matrix $\vec{A}\in\mathbb{R}^{m\times n}$,
$m\geq n$, its condition number in any $p$-norm, denoted by $\kappa_{p}\left(\vec{A}\right)$,
is bounded by the ratio of the largest and smallest column vectors
in p-norm, i.e., 
\begin{align}
\kappa_{p}\left(\vec{A}\right) & \geq\frac{\max_{1\leq i\leq n}\left\Vert \vec{a}_{i}\right\Vert _{p}}{\min_{1\leq j\leq n}\left\Vert \vec{a}_{j}\right\Vert _{p}},\label{eq:cond_est}
\end{align}
where $\vec{a}_{k}$ denotes the $k$th column vector of $\vec{A}$.
\end{prop}
The above fact offers an intuitive explanation of a source of ill-conditioning
in traditional finite element methods due to poorly shaped elements:
poorly shaped elements may lead to unbounded large entries in local
stiffness matrices, so the column norms of the global stiffness matrix
would vary substantially, and in turn the global stiffness matrix
is necessarily ill-conditioned. In the context of AES-FEM, there can
be two potential sources of local instability due to poor scaling.
First, the unnormalized local Vandermonde system given in (\ref{eq:Vc_equal_f})
is in general very poorly scaled. We resolved this by normalizing
the Vandermonde system to avoid poor scaling. Second, the normalized
Vandermonde system may still be ill-conditioned occasionally, when
a stencil is degenerate or nearly degenerate, which could lead to
unbounded large values in the local stiffness matrix. We resolved
this issue by using QR with column pivoting and condition-number estimation. 

Even if the local stiffness matrices are bounded, the global stiffness
matrix may still be ill-conditioned due to linearly dependent rows
or columns. This is the potential source of instability for some meshless
methods when two points nearly coincide; the two points may share
the same stencil and basis functions, so that the rows or columns
corresponding to the two points would be nearly identical. Therefore,
these meshless methods also require good point distributions. In AES-FEM,
we utilize the mesh topology to construct the stencil, as we will
describe in Section \ref{sec:Implementation}. This ensures that no
two vertices share the same stencil unless there are coincident points,
and hence it gives a strong guarantee that the rows in the global
stiffness matrix are linearly independent.

The aforementioned reasons are the most common causes of instability
for solving elliptic PDEs. Another source of instability is a cluster
of coincident points or inverted elements, which rarely happen in
practice, and hence we defer their treatments to future work. As we
will demonstrate numerically in Section~\ref{sec:Results}, by resolving
these instabilities, AES-FEM produces well-conditioned stiffness matrices
for meshes, even with very bad quality elements or point distributions.

\section{Implementation \label{sec:Implementation}}

We discuss the practical aspects of the implementation of AES-FEM
in this section. We start with a discussion of the utilized mesh data
structure and then explain how this enables quick and efficient neighborhood
selection. Finally, the algorithms are presented and runtime is analyzed.

\subsection{Data Structure \label{sub:Data-Structure}}

We use an Array-based Half-Facet (AHF) data structure \cite{Dyedov2014}
to store the mesh information. In a $d$-dimensional mesh, the term
\emph{facet} refers to the $(d-1)$-dimensional mesh entities; that
is, in 2D the facets are the edges, and in 3D the facets are the faces.
The basis for the half-facet data structure is the idea that every
facet in a manifold mesh is made of two half-facets oriented in opposite
directions. We refer to these two half-facets as \emph{sibling half-facets}.
Half-facets on the boundary of the domain have no siblings. In 2D
and 3D, the half-facets are \emph{half-edges} and \emph{half-faces},
respectively. We identify each half-facet by a two tuple: the element
ID and a local facet ID within the element. In 2D, we store the element
connectivity, sibling half-edges, and a mapping from each node to
an incident half-edge. In 3D, we store the element connectivity, sibling
half-faces, and a mapping from each node to an incident half-face.
For an example of a 2D mesh and the associated data structure, see
Figure~\ref{fig:heds_square}. This data structure allows us to do
neighborhood queries for a node in constant time (provided the valance
is bounded). For additional information about the AHF data structure,
see \cite{Dyedov2014}.

\begin{figure}
\begin{centering}
\includegraphics[scale=0.75]{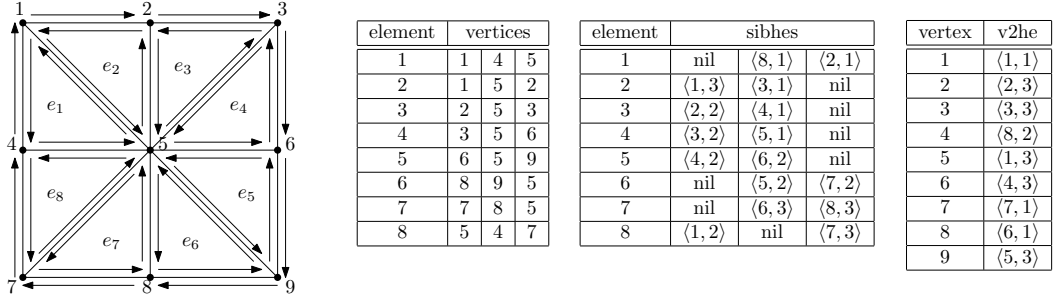}
\par\end{centering}

\caption{An example of half edges and associated data structure.\label{fig:heds_square}}
\end{figure}

\subsection{Neighborhood Selection \label{sub:Neighborhood-Selection}}

The use of the AHF data structure allows us to quickly find the neighborhood
of a node. We use the concept of rings to control the size of the
neighborhood. The \textit{1-ring neighbor elements} of a node are
defined to be the elements incident on the node. The \textit{1-ring
neighborhood} of a node contains the nodes of its 1-ring neighbor
elements \cite{Jiao2008}. Most of the time, when using GFD with second
order basis functions or when constructing second order GLP basis
functions, the 1-ring neighborhood of a node supplies the appropriate
number of nodes. If the valance is low, it might be necessary to further
expand and collect more nodes for the neighborhood. Therefore, for
any integer $k\geq1$, we define the $(k+1)$-\textit{ring neighborhood}
as the nodes in the $k$-ring neighborhood plus their 1-ring neighborhoods. 

As $k$ increases, the average size of the $k$-ring neighborhood
grows very quickly. The granularity can be fine-tuned by using fractional
rings. In 2D we use half rings, which are defined in \cite{Jiao2008};
for any integer $k\geq1$ the $(k+\sfrac{1}{2})$\textit{-ring neighborhood}
is the $k$-ring neighborhood plus the nodes of all the faces that
share an edge with the $k$-ring neighborhood. See Figure~\ref{fig:stencil}
for a visualization of rings and half-rings in 2D. We extend this
definition to 3D and introduce $\sfrac{1}{3}$- and $\sfrac{2}{3}$-rings.
For any integer $k\geq1$, the $\left(k+\sfrac{1}{3}\right)$-\textit{ring
neighborhood} contains the $k$-ring neighborhood plus the nodes of
all elements that share a face with the $k$-ring neighborhood. The
$(k+\sfrac{2}{3})$-\textit{ring neighborhood} contains the $k$-ring
neighborhood plus the nodes of all faces that share an edge with the
$k$-ring neighborhood.

\begin{figure}
\begin{centering}
\includegraphics[scale=0.5]{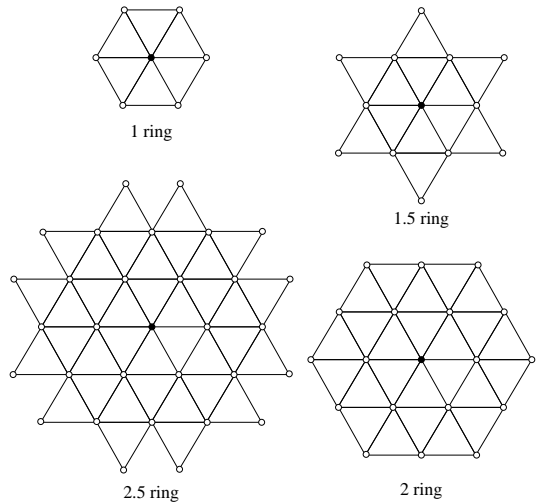}
\par\end{centering}

\caption{Examples of 2D stencils with 1-ring, 1.5-ring, 2-ring, and 2.5-ring
neighborhoods of center node (in solid black).\label{fig:stencil}}
\end{figure}

Note that for 2D triangular and 3D tetrahedral meshes, the 1-ring
neighborhood typically has enough points for constructing quadratic
GLP basis functions. Therefore, the stiffness matrix from AES-FEM
has a similar sparsity pattern to that from standard FEM with linear
shape functions. However, when the 1-ring neighborhood is too small,
the extended stencil with a larger ring allows AES-FEM to overcome
mesh-quality dependence and improve its local stability.

\subsection{Stable Computation of GLP Basis Functions \label{sub:Stable_Comp_of_GFD}}

Even with proper neighborhood selection, it cannot be guaranteed that
the least-squares problem in (\ref{eq:WLS}) will be well-conditioned.
Hence, it is critical to use a robust method that ensures the accuracy
and stability of the approximate solutions. 

Note that one standard technique in linear algebra for solving rank-deficient
least-squares problems is truncated singular value decomposition (TSVD)
\cite{Golub13MC}. The TSVD is not recommended here, because it can
result in the loss of partition of unity of the basis functions. The
reason is as follows. When truncating SVD, one truncates any singular
value $\sigma_{j}$ that is smaller than $\epsilon\sigma_{1}$, where
$\sigma_{1}$ is the largest singular value and $\epsilon$ is some
small positive value, such as $10^{-4}$. These singular values may
be necessary for computing the constant terms in the GLP basis functions,
and hence their loss can result in a set of basis functions that lack
the partition of unity and in turn may compromise convergence.

We avoid the above issue by using truncated QR factorization with
column pivoting (QRCP). When performing QRCP, one can find a numerical
rank $r$ of matrix $\vec{R}$ so that the condition number $\kappa(\vec{R}_{1:r,1:r})<1/\epsilon$,
where $\epsilon$ is some small positive value, such as $10^{-4}$.
This is elaborated in the below discussion of Algorithm~\ref{alg:InitializationCVM}.
If $r$ is less than the size of $\vec{R}$, then any diagonal entry
of $\vec{R}$ in position $r+1$ or greater is truncated, thus truncating
the $(r+1)$th and subsequent columns of $\vec{Q}$. In QRCP, we require
the first column not to be permuted, and this ensures the resulting
basis functions satisfy the property of partition of unity.

In Algorithm~\ref{alg:InitializationCVM}, we present the procedure
for initializing the generalized Vandermonde matrix $\tilde{\vec{V}}$
and factoring it using QRCP. The generalized Vandermonde matrix is
formed from the local coordinates of the stencils and is scaled by
the column scaling matrix $\vec{S}$ and the row scaling matrix $\vec{W}$.
The resulting matrix is then factored using QRCP. We use a variant
of Householder triangularization \cite{Golub13MC} since this procedure
is more efficient and stable than alternatives (such as Gram-Schmidt
orthogonalization). When implementing this procedure, the QR factorization
of $\tilde{\vec{V}}$ can overwrite $\vec{V}$. The $j\text{th}$
Householder reflection vector is of size $n-j+1$. By requiring the
first element of the vector to be positive, the first element may
be reconstructed from the other elements, and thus only $n-j$ entries
are required to store the $j\text{th}$ Householder reflection vector.
The Householder vectors are stored in the lower triangular part of
$\vec{V}$ and the $\vec{R}$ entries are stored in the upper part.
The permutation matrix $\vec{P}$ is stored in a permutation vector.
\begin{algorithm}
\caption{\label{alg:InitializationCVM}Initialization of a Generalized Vandermonde
Matrix}

\textbf{function}: initiate\_GVM

\textbf{input}: 1. $\vec{x}_{k}$: local coordinates of stencil

\hspace{1.2cm}2. $\vec{w}$: vector of row weights

\hspace{1.2cm}3. $p$: desired degree for $\vec{V}$ 

\begin{raggedright}
\hspace{1.2cm}4. $\epsilon$: tolerance for rank deficiency
\par\end{raggedright}

\begin{raggedright}
\textbf{output}: struct gvm: with fields $\vec{W}$, $\vec{S}$, $\vec{Q}$,
$\vec{R}$, $\vec{P}$, $r$ (estimated rank)
\par\end{raggedright}

\begin{algorithmic}[1]

\STATE create generalized Vandermonde matrix $\vec{V}$ from local
coordinates $\vec{x}_{k}$

\STATE determine column scaling matrix $\vec{S}$

\STATE $\vec{W}\leftarrow\text{diag}(\vec{w})$ 

\STATE $\tilde{\vec{V}}\leftarrow\vec{W}\vec{V}\vec{S}$

\STATE solve $\tilde{\vec{V}}\vec{P}=\vec{Q}\vec{R}$

\STATE estimate rank $r$ from $\vec{R}$ so that $r=\max\left\{ i|\text{cond}\left(\vec{R}_{1:i,1:i}\right)\leq1/\epsilon\right\} $

\end{algorithmic}
\end{algorithm}

In addition to computing the QR factorization of the generalized Vandermonde
matrix, an estimation of the numerical rank of $\vec{R}$ is also
computed in Algorithm~\ref{alg:InitializationCVM}. The rank is important
for ensuring the overall stability of other algorithms that use this
initialization step. In order to estimate the numerical rank, we estimate
the condition numbers of the leading principal sub-matrices of $\vec{R}$,
$\vec{R}_{1:r,1:r}$, and find the largest $r$ such that $\tilde{\kappa}(\vec{R}_{1:r,1:r})\leq1/\epsilon$
where $\tilde{\kappa}$ is the estimated condition number and $\epsilon$
is some given drop-off tolerance depending on the degree of polynomials.
 Note that since the matrix $\vec{R}$ is small, the condition numbers
in different norms differ by only a small factor. Therefore for efficiency,
we estimate the condition number of $\vec{R}$ in the 1-norm using
the algorithm described in \cite{higham1987survey}.

Once the generalized Vandermonde matrix has been initialized, it may
be used to construct generalized finite differentiation operators
from the weighted least squares approximations, as described in Algorithm~\ref{alg:Implicit-differentiation}.
The input for this algorithm is the output from Algorithm~\ref{alg:InitializationCVM}
and a vector $\vec{a}$. The vector $\vec{a}=\mathcal{D}\vec{\mathcal{P}}(\vec{x})$
contains the values for some specified derivative $\mathcal{D}$ of
the monomial basis functions $\vec{\mathcal{P}}$ at point $\vec{x}$.
For example, let $\mathcal{D}$ be $\frac{\partial}{\partial y}$.
Then in 2D, we have $\vec{a}=\mathcal{D}\vec{\mathcal{P}}(x,y)=[0\ 0\ 1\ 0\ x\ 2y]^{T}$
and in 3D, we have $\vec{a}=\mathcal{D}\vec{\mathcal{P}}(x,y,z)=[0\ 0\ 1\ 0\ 0\ x\ 0\ 2y\ 0\ 0]^{T}$.
The algorithm returns a vector of weights $\vec{d}$ so that $\vec{d}^{T}\vec{g}=\mathcal{D}f(\vec{x})$
for a vector $\vec{g}=\left[f_{1}\ f_{2}\ \dots\ f_{m}\right]^{T}$
containing the values of the function at the stencil points. Note
that for a GLP basis function, the returned weights are the values
of the specified derivative at the points in the stencil. 

\begin{algorithm}
\caption{\label{alg:Implicit-differentiation}Approximating $\mathcal{D}f$
at given point $\vec{x}$ from WLS}

\begin{raggedright}
\textbf{function} diff\_WLS
\par\end{raggedright}

\begin{raggedright}
\textbf{input}: 1. struct gvm: with fields $\vec{W}$, $\vec{S}$,
$\vec{Q}$, $\vec{R}$, $\vec{P}$, $r$ (estimated rank)
\par\end{raggedright}

\begin{raggedright}
\hspace{1.0cm}2. coefficients $\vec{a}=\mathcal{D}\mathcal{P}(\vec{x})$
\par\end{raggedright}

\begin{raggedright}
\textbf{output}: weights $\vec{d}$, so that $\vec{d}^{T}\vec{g}=\mathcal{D}f(\vec{x})$
for $\vec{g}$ containing $f(\vec{x}_{k})$ at stencil points
\par\end{raggedright}

\begin{algorithmic}[1]

\STATE $\vec{a}\leftarrow\left(\vec{P}_{:,1:r}\right)^{T}\vec{S}^{-1}\vec{a}$;

\STATE $\vec{a}\leftarrow\vec{R}_{1:r,1:r}^{-T}\vec{a}$;

\STATE $\vec{d}\leftarrow\vec{W}\vec{Q}_{:,1:r}\vec{a}$;

\end{algorithmic}
\end{algorithm}

In terms of the computational cost, the step that dominates Algorithm~\ref{alg:InitializationCVM}
is the QRCP factorization, which takes $\mathcal{O}\left(2mn^{2}-\frac{2}{3}n^{3}\right)$
flops where $\tilde{\vec{V}}$ is $m\times n$ \cite{Golub13MC}.
Here, $m$ is the number of points in the stencil and $n$ is the
number of terms in the Taylor series expansion (for second order expansion
$n=6$ in 2D and $n=10$ in 3D). As long as the valance is bounded,
that is $m$ is bounded, this algorithm is executed in a constant
time. Compared to calculations based on the standard finite-element
basis functions, which are tabulated, the computation based on the
GLP basis functions is more expensive. This leads to higher cost of
AES-FEM in assembling the stiffness matrix and load vector, as we
discuss next. However, this cost is a small constant per element,
and AES-FEM can be more efficient overall by delivering higher accuracy,
as we will demonstrate in Section~\ref{sec:Results}.

\subsection{Assembly of Stiffness Matrix and Load Vector \label{sub:Algorithms}}

Algorithm~\ref{alg:AESFEM} presents a summary of the AES-FEM procedure
for assembling the stiffness matrix and load vector for a PDE with
Dirichlet boundary conditions. Unlike the standard FEM procedure,
we build the stiffness matrix row by row, rather than element by element.
This is because the most computationally expensive part of the procedure
is to compute the derivatives for the set of basis functions on each
stencil. A weight function is nonzero only on the neighborhood around
its corresponding node. Since the weight functions correspond to the
rows, we assemble the stiffness matrix row by row, ensuring that we
will only need to compute the derivatives for each neighborhood once.

When computing a row of the stiffness matrix, the first step is to
obtain the stencil of node $k$. This step is performed by utilizing
the data structure presented in Subsection \ref{sub:Data-Structure}
and the proper size of the stencil is ensured by choosing the ring
sizes adaptively. Next, the local coordinates are calculated for the
points in the stencils and the row weights are computed. Using Algorithm~\ref{alg:InitializationCVM},
the QR factorization of the generalized Vandermonde matrix is computed
for the neighborhood. Then for each element that contains node $k$,
we perform the integration of the weak form in a manner that is similar
to standard FEM. The element Jacobian is computed and used to find
the local coordinates of the quadrature points. The derivatives of
the weight function are computed, that is $\nabla\psi_{i}$, at the
quadrature points of the current element. Recall that the weight functions
are the standard hat functions. Next the derivatives of the basis
functions, that is $\nabla\phi_{j}$, are computed at the quadrature
points of the current element. The basis functions are the GLP basis
functions and thus Algorithm~\ref{alg:Implicit-differentiation}
is used. The value of the integral on the current element is computed
and either added to the stiffness matrix or subtracted from the load
vector, depending on whether the basis function corresponds to a node
with Dirichlet boundary conditions. 

When computing the load vector, typically the entries $b_{i}=\int\psi_{i}f\ dV$
are computed using a quadrature rule.  One may evaluate $f$ at the
quadrature points in two ways. The first way is to use the standard
procedure in FEM, i.e., to use the FEM basis functions and the values
of $f$ at the nodes of the element. Let $\vec{M}$ be the vector
of FEM shape functions evaluated at quadrature point $\vec{x}_{k}$
and $\vec{g}_{\mathrm{elem}}$ be the vector of the function values
at the nodes of the element. Then, we have 
\begin{equation}
f\left(\vec{x}_{k}\right)=\vec{M}\cdot\vec{g}_{\mathrm{elem}}.
\end{equation}
Alternatively, we may use GLP basis functions to interpolate the values
of $f$ at the quadrature points. We can approximate an arbitrary
function using the set of GLP basis functions. In matrix notation,
we have 
\begin{equation}
f\left(\vec{x}_{k}\right)=\vec{g}_{\mathrm{sten}}^{T}\left(\vec{S}\tilde{\vec{V}}^{+}\vec{W}\right)^{T}\vec{D}\vec{\mathcal{P}}\left(\vec{x}_{k}\right),
\end{equation}
where $\vec{g}_{\mathrm{sten}}$ is the vector of function values
at the nodes in the stencil and the vector $\vec{\mathcal{P}}\left(\vec{x}_{k}\right)$
has been evaluated at the quadrature point $\vec{x}_{k}$. As mentioned
earlier, we refer to the variant of AES-FEM using the former method
of calculating the load vector as AES-FEM 1 and refer to the latter
variant as AES-FEM 2.

We use Gaussian quadrature to perform the integration within each
element. For quadratic GLP basis functions in 2D, we use a 1-point
rule for stiffness matrix, and a 3-point rule for the load vector.
In 3D, we use a 1-point rule for the stiffness matrix and a 4-point
rule for the load vector. These rules are exact because the basis
functions and their derivatives are quadratic and linear, respectively.

It is worth noting that because of the properties of generalized Lagrange
polynomial basis functions, Dirichlet boundary conditions may be imposed
in AES-FEM in the same manner as in standard FEM. One does not need
to use Lagrange multipliers or a penalty method. Additionally, the
standard method for imposing Neumann boundary conditions may be used
in AES-FEM. 

All the steps inside of the primary for-loop are executed in constant
time, assuming that the size of each neighborhood is bounded. Therefore,
the assembly of the stiffness matrix in AES-FEM has an asymptotic
runtime of $\mathcal{O}(n)$, where $n$ is the number of nodes in
the mesh. When using AES-FEM 2, that is when WLS approximation is
used to compute the approximation of $f$ at the quadrature points,
Algorithm \ref{alg:Implicit-differentiation} is called again. While
this function is constant in runtime, it has a large coefficient and
thus takes longer than approximating the values of $f$ using FEM
(hat) basis functions. Therefore, the assembly time for AES-FEM 2
is longer than that for AES-FEM 1, as can be seen in Section~\ref{sec:Results}.

\begin{algorithm}
\caption{\label{alg:AESFEM}Building a Stiffness Matrix and Load Vector using
AES-FEM}

\textbf{function}: aes\_fem

\textbf{input}: 1. $\vec{x}$, $\mathtt{elem,\ opphfs,\ vh2f}$: mesh
information

\hspace{1.2cm}2. $p$: desired degree for GLP functions

\hspace{1.2cm}3. $\epsilon$: tolerance for rank deficiency

\begin{raggedright}
\hspace{1.2cm}4. $\mathtt{AESFEM1}$: boolean for AES-FEM 1 or AES-FEM
2
\par\end{raggedright}

\begin{raggedright}
\hspace{1.2cm}5. $\mathtt{isDBC}$: flags for Dirichlet boundary
conditions
\par\end{raggedright}

\begin{raggedright}
\textbf{output}: stiffness matrix $\vec{K}$ and load vector $\vec{b}$
\par\end{raggedright}

\begin{algorithmic}[1]

\FOR {each node without Dirichlet boundary conditions} 

\STATE obtain neighborhood of node 

\STATE calculate local parameterization $\vec{x}_{k}$ and row weights
$\vec{w}$ for neighborhood

\STATE aes\_gvm $\leftarrow$ initiate\_GVM($\vec{x}_{k},$ $\vec{w},$
$p,$ $\epsilon$)

\STATE obtain local element neighborhood

\FOR {each element in local neighborhood} 

\STATE calculate element Jacobian and local coordinates of quad-points

\STATE calculate derivatives of FEM shape functions at quad-points

\STATE $\vec{a}\leftarrow\mathcal{D}\vec{\mathcal{P}}(\vec{x})$
where $\mathcal{D}\vec{\mathcal{P}}(\vec{x})$ is defined by the PDE
we are solving

\STATE GLPderivs $\leftarrow$ diff\_WLS(aes\_gvm, $\vec{a}$)

\FOR {each node in neighborhood} 

\IF {not Dirichlet BC node} 

\STATE add integral to appropriate stiffness matrix entry

\ELSE

\STATE subtract integral from load vector

\ENDIF

\ENDFOR

\IF {$\mathtt{AESFEM1}$} 

\STATE calculate load vector over current element using FEM approximations
for quad-points

\ELSE

\STATE calculate load vector over current element using GLP approximations
for quad-points

\ENDIF

\ENDFOR

\ENDFOR

\end{algorithmic}
\end{algorithm}

Once the stiffness matrix and the load vector are assembled, we use
the generalized minimal residual method (GMRES) \cite{saad1986gmres}
with a preconditioner to solve the linear system. GMRES is a standard
Krylov subspace method for iteratively solving a sparse, nonsymmetric
linear system. Specifically, we use the MATLAB implementation of GMRES.
Another option for solving this system would be multigrid methods
\cite{trottenberg2000multigrid}, and we will explore their use in
future work. As a preconditioner, we use incomplete LU factorization
(ILU) in 2D and Gauss-Seidel in 3D.

Finally, for completeness, we include Algorithm~\ref{alg:GFD} to
summarize the GFD procedure for solving PDEs with Dirichlet boundary
conditions. We use this algorithm primarily for comparison with the
AES-FEM algorithm. We can see that for GFD, we need to use Algorithm~\ref{alg:Implicit-differentiation}
to compute the weights once for each non-Dirichlet node in the mesh;
that is, if there are $n$ non-Dirichlet nodes, Algorithm~\ref{alg:Implicit-differentiation}
is called $n$ times. For AES-FEM, for a given node, we use this algorithm
once for every element containing that node. Thus if every node has
a neighborhood of $k$ elements and there are $n$ non-Dirichlet nodes,
we call the algorithm $kn$ times. Therefore, the assembly time is
longer for AES-FEM than for GFD, as we will see in Section~\ref{sec:Results}. 

\begin{algorithm}
\caption{\label{alg:GFD}Constructing a GFD coefficient matrix}

\textbf{function}: gfd

\textbf{input}: 1. $\vec{x}$, $\mathtt{elem,\ opphfs,\ vh2f}$: mesh
information

\hspace{1.2cm}2. $p$: desired degree for GFD functions

\begin{raggedright}
\hspace{1.2cm}3. $\mathtt{isDBC}$: flags for Dirichlet boundary
conditions
\par\end{raggedright}

\begin{raggedright}
\textbf{output}: GFD matrix $\vec{K}$ and vector $\vec{b}$
\par\end{raggedright}

\begin{algorithmic}[1]

\FOR {each node without Dirichlet boundary conditions} 

\STATE obtain neighborhood of node

\STATE calculate local parameterization and row weights for neighborhood

\STATE gfd\_cvm $\leftarrow$ initiate\_CVM($\vec{x}_{k},$ $\vec{w},$
$p,$ $\epsilon$)

\STATE $\vec{a}\leftarrow\mathcal{D}\vec{\mathcal{P}}(\vec{x})$
where $\mathcal{D}\vec{\mathcal{P}}(\vec{x})$ is defined by the PDE
we are solving

\STATE GFDderivs $\leftarrow$ diff\_WLS(gfd\_cvm, $\vec{a}$)

\FOR {each node in local neighborhood} 

\IF {not BC node} 

\STATE enter value in matrix

\ELSE

\STATE subtract from RHS vector

\ENDIF

\ENDFOR

\ENDFOR

\end{algorithmic}
\end{algorithm}

\section{Numerical Results\label{sec:Results} }

In this section, we compare the accuracy, efficiency, and element
shape quality dependence of AES-FEM 1, AES-FEM 2, FEM with linear
basis functions, and GFD with quadratic basis functions. We compare
these four method because the sparsity pattern of the coefficient
matrix is nearly identical for all the methods. The sparsity pattern
determines the amount of storage necessary and also the computational
cost of vector-matrix multiplication. The errors are calculated using
the discrete $L_{2}$ and $L_{\infty}$ norms. Let $u$ denote the
exact solution and let $\hat{u}$ denote the numerical solution. Then,
we calculate the norms as
\begin{equation}
L_{2}\left(\text{error}\right)=\left(\int_{\Omega}|\hat{u}-u|^{2}\partial\Omega\right)^{1/2}\quad\mbox{and}\quad L_{\infty}\left(\text{error}\right)=\max_{i}|\hat{u}-u|.
\end{equation}
For each series of meshes of different grid resolution, we calculate
the average convergence rate as 
\begin{equation}
\text{convergence rate }=-\log_{2}\left(\frac{\text{error of mesh 1}}{\text{error of mesh 4}}\right)\left/\log_{2}\left(\sqrt[d]{\frac{\text{nodes in mesh 1}}{\text{nodes in mesh 4}}}\right),\right.
\end{equation}
where $d$ is the spacial dimension.

\subsection{2D Results}

In this section, we present the results of our 2-dimensional experiments.
We use two different series of meshes, each series with 4 meshes.
The first series of meshes (referred to collectively as ``mesh series
1'') is generated by placing nodes on a regular grid and then using
MATLAB's Delaunay triangularization function to create the elements.
The meshes range in size from $64\times64$ to $512\times512$ nodes.
On the most refined mesh, the minimum angle is 45 degrees and the
maximum angle is 90 degrees. The maximum aspect ratio is 1.41, where
a triangle's aspect ratio is defined as the ratio of the length of
longest edge to the length of the smallest edge. The second series
of meshes (referred collectively as ``mesh series 2'') is generated
by using Triangle \cite{shewchuk1996triangle}. The number of nodes
for each level of refinement is 4,103, 16,401, 65,655, and 262,597,
respectively, approximately the same as those in series 1. On the
most refined mesh, the maximum angle is 129.6 degrees and the minimum
angle is 22.4 degrees. The maximum aspect ratio is 2.61. See Figure~\ref{fig:Mesh_2D_example}
for a visualization of the types of meshes used. 

\begin{flushleft}
\begin{figure}
\begin{minipage}[t]{0.43\textwidth}%
\begin{center}
\includegraphics[width=1\textwidth]{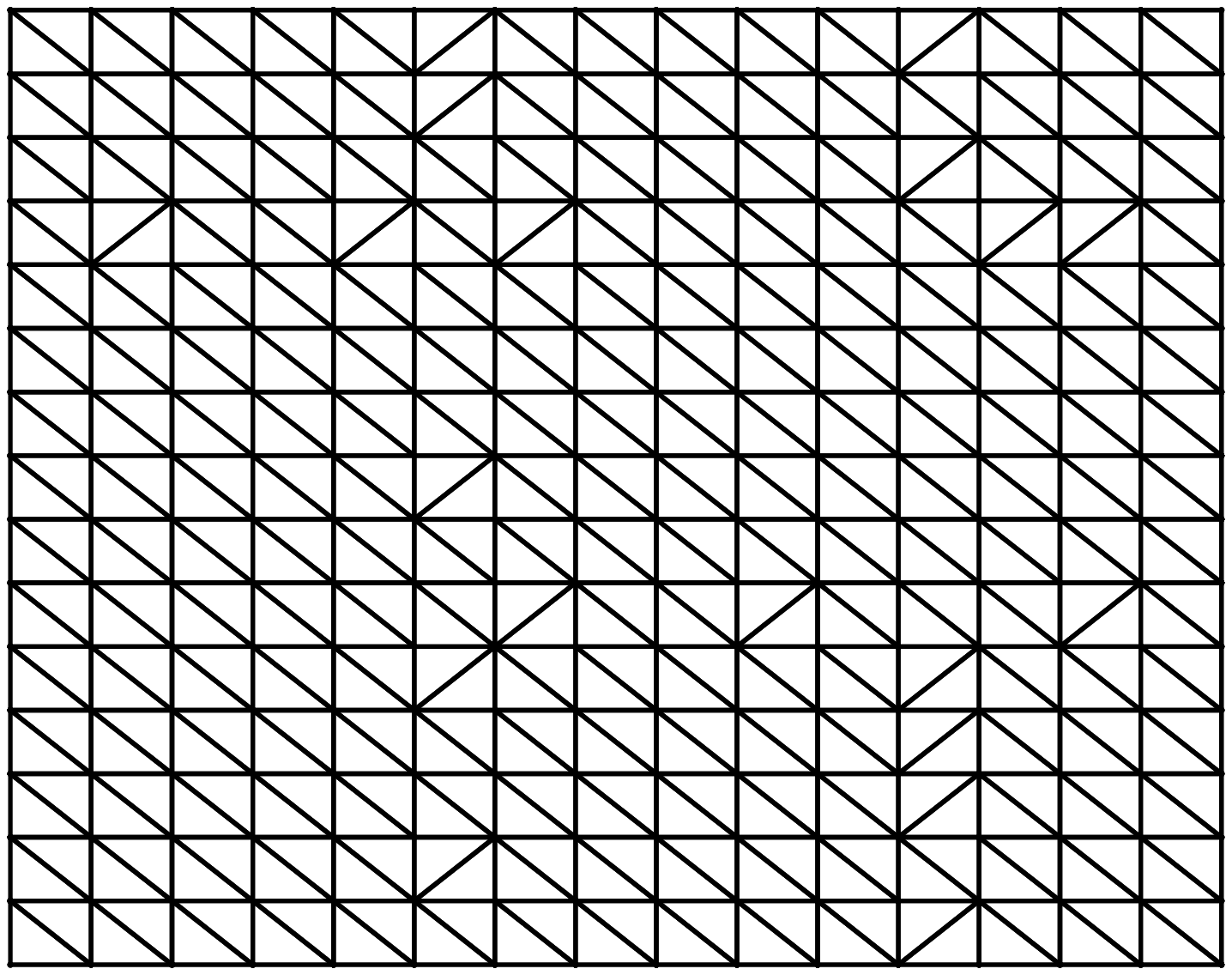}
\par\end{center}%
\end{minipage}\hfill{} %
\begin{minipage}[t]{0.43\textwidth}%
\begin{center}
\includegraphics[width=1\textwidth]{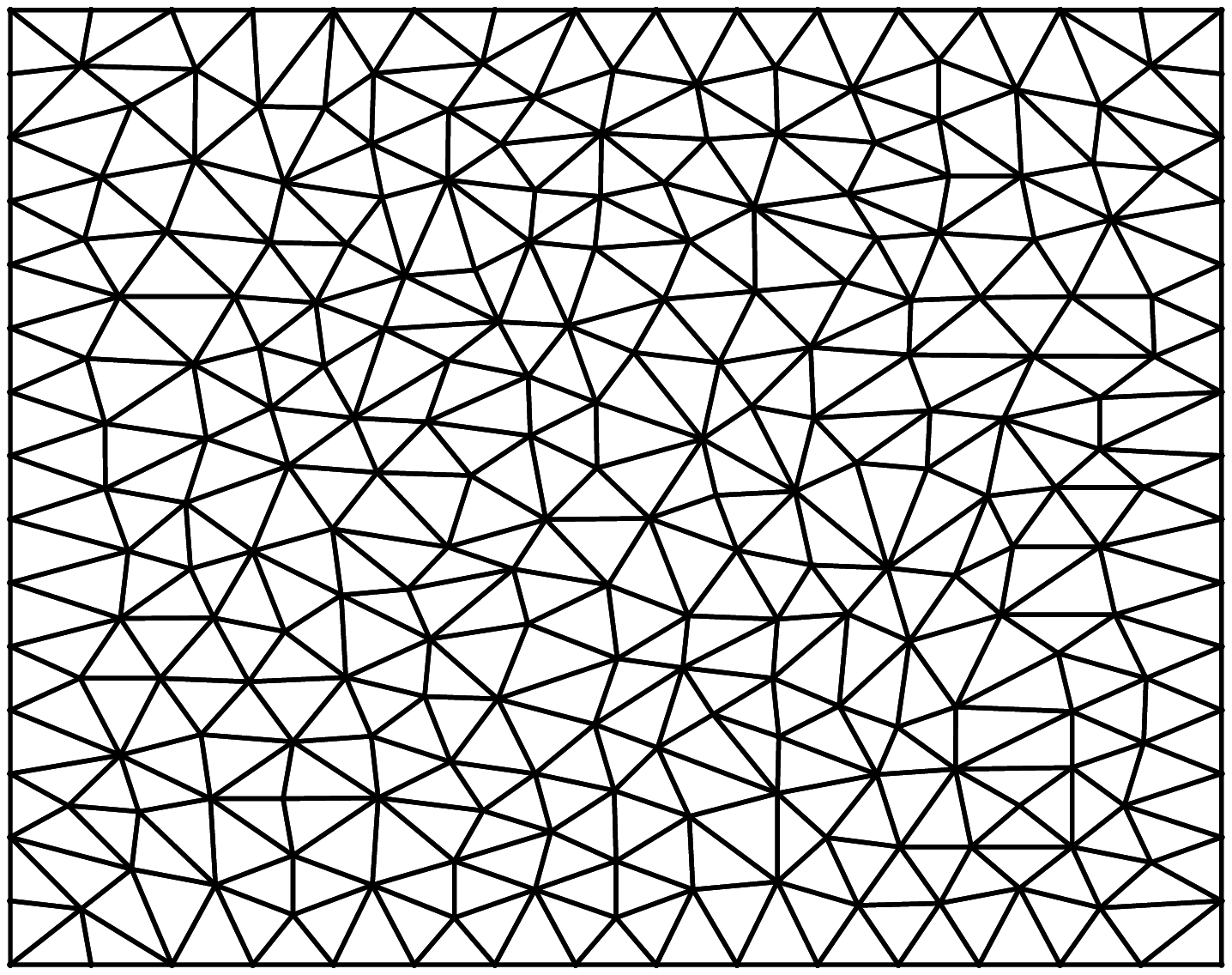}
\par\end{center}%
\end{minipage}

\raggedright{}\caption{The mesh on the left is representative of the meshes used in series
1. The mesh on the right is representative of the meshes used in series
2. Note that the meshes above are coarser than the meshes used in
computations so that the details can be seen clearly.\label{fig:Mesh_2D_example}}
\end{figure}

\par\end{flushleft}

\subsubsection{Poisson Equation}

The first set of results we present is for the Poisson equation with
Dirichlet boundary conditions on the unit square. That is, 
\begin{align}
-\nabla^{2}u & =f\quad\text{in }\Omega=[0,1]^{2},\\
u & =g\quad\text{on }\partial\Omega.
\end{align}
We consider the following three analytic solutions:
\begin{align}
u_{1} & =16x(1-x)y(1-y),\label{eq:polynomial_solution_2D}\\
u_{2} & =\cos(\pi x)\cos(\pi y),\label{eq:trigonometric_solution_2D}\\
u_{3} & =\frac{1}{\sinh\pi\cosh\pi}\sinh(\pi x)\cosh(\pi y).\label{eq:hyperbolic_solution_2D}
\end{align}
The Dirichlet boundary conditions are obtained from the given analytic
solutions. The boundary conditions for $u_{1}$ are homogeneous and
they are non-homogenous for $u_{2}$ and $u_{3}$ . 

The $L_{\infty}$ and $L_{2}$ norm errors for $u_{1}$ on mesh series
1 are displayed in Figure \ref{fig:2D_Poisson_mesh1_u1}. One can
see that the two graphs are fairly similar; this is true for $u_{2}$
and $u_{3}$ as well and thus we show only the $L_{\infty}$ norm
errors for these two problems; see Figure \ref{fig:2D_Poisson_mesh1_u2=0000263}.
GFD is the most accurate for $u_{1}$ and $u_{2}$ and AES-FEM 2 is
the most accurate for $u_{3}$. FEM is the least accurate in all three
cases. 

For mesh series 2, the $L_{\infty}$ and $L_{2}$ norm errors for
$u_{1}$ can be seen in Figure \ref{fig:2D_Poisson_mesh2_u1} and
the $L_{\infty}$ norm errors for $u_{2}$ and $u_{3}$ can be seen
in Figure \ref{fig:2D_Poisson_mesh2_u2=0000263}. On this mesh series,
AES-FEM 2 has the lowest error for $u_{1}$ and $u_{2}$. For $u_{3}$,
the errors for GFD and AES-FEM 2 are very similar. 

\begin{flushleft}
\begin{figure}
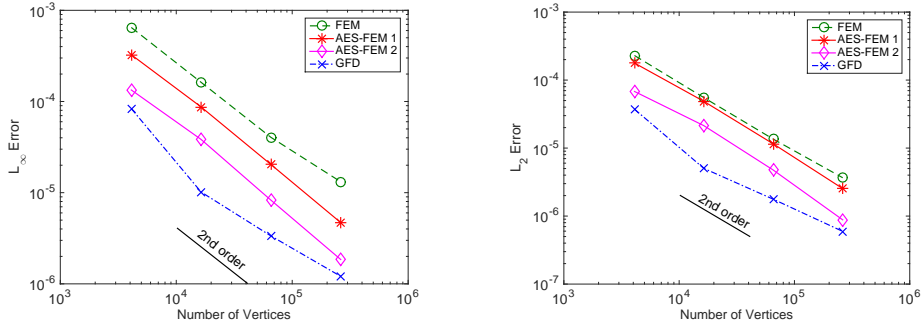

\begin{minipage}[t]{0.45\textwidth}%
\begin{center}
\includegraphics[width=1\textwidth]{2D_Pois_mesh1_u1_Inf}
\par\end{center}%
\end{minipage}\hfill{} %
\begin{minipage}[t]{0.45\textwidth}%
\begin{center}
\includegraphics[width=1\textwidth]{2D_Pois_mesh1_u1_L2}
\par\end{center}%
\end{minipage}

\raggedright{}\caption{The errors for 2D Poisson equation on mesh 1 for $u_{1}$. The errors
were computed using the $L_{\infty}$ norm (left) and the $L_{2}$
norm (right). \label{fig:2D_Poisson_mesh1_u1}}
\end{figure}

\par\end{flushleft}

\begin{flushleft}
\begin{figure}
\begin{minipage}[t]{0.45\textwidth}%
\begin{center}
\includegraphics[width=1\textwidth]{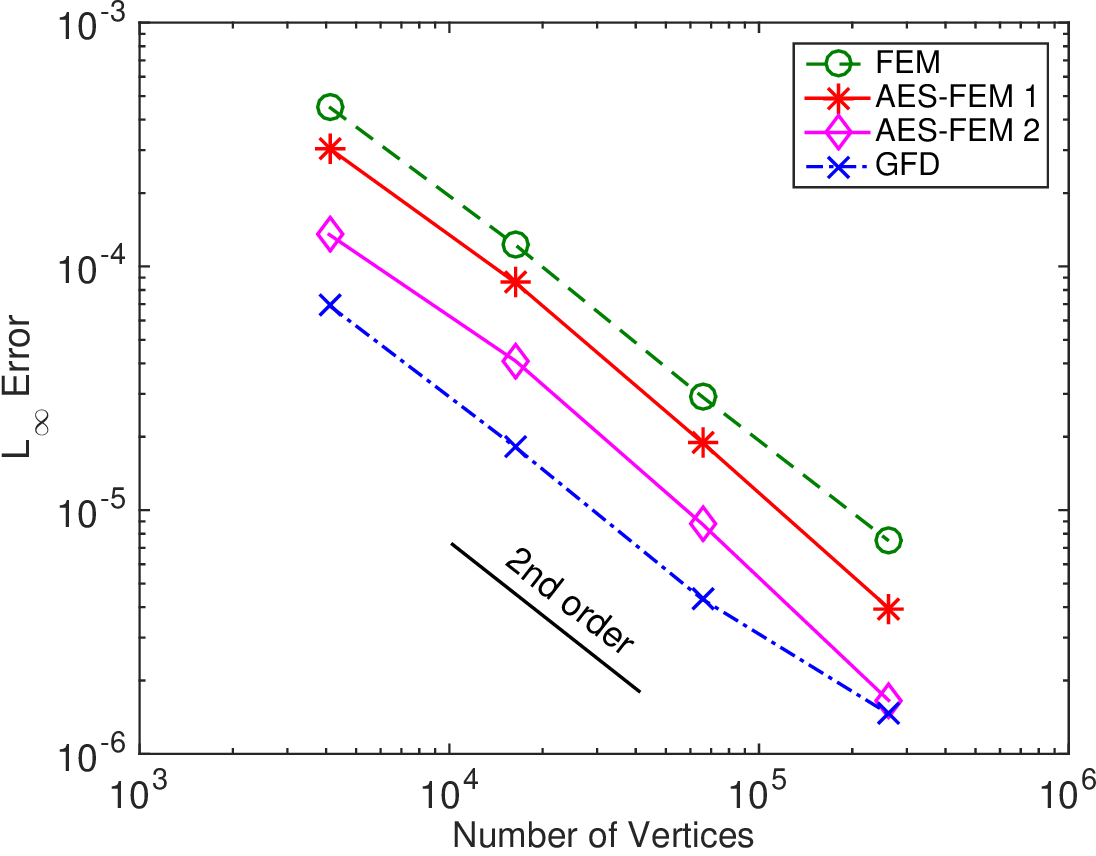}
\par\end{center}%
\end{minipage}\hfill{} %
\begin{minipage}[t]{0.45\textwidth}%
\begin{center}
\includegraphics[width=1\textwidth]{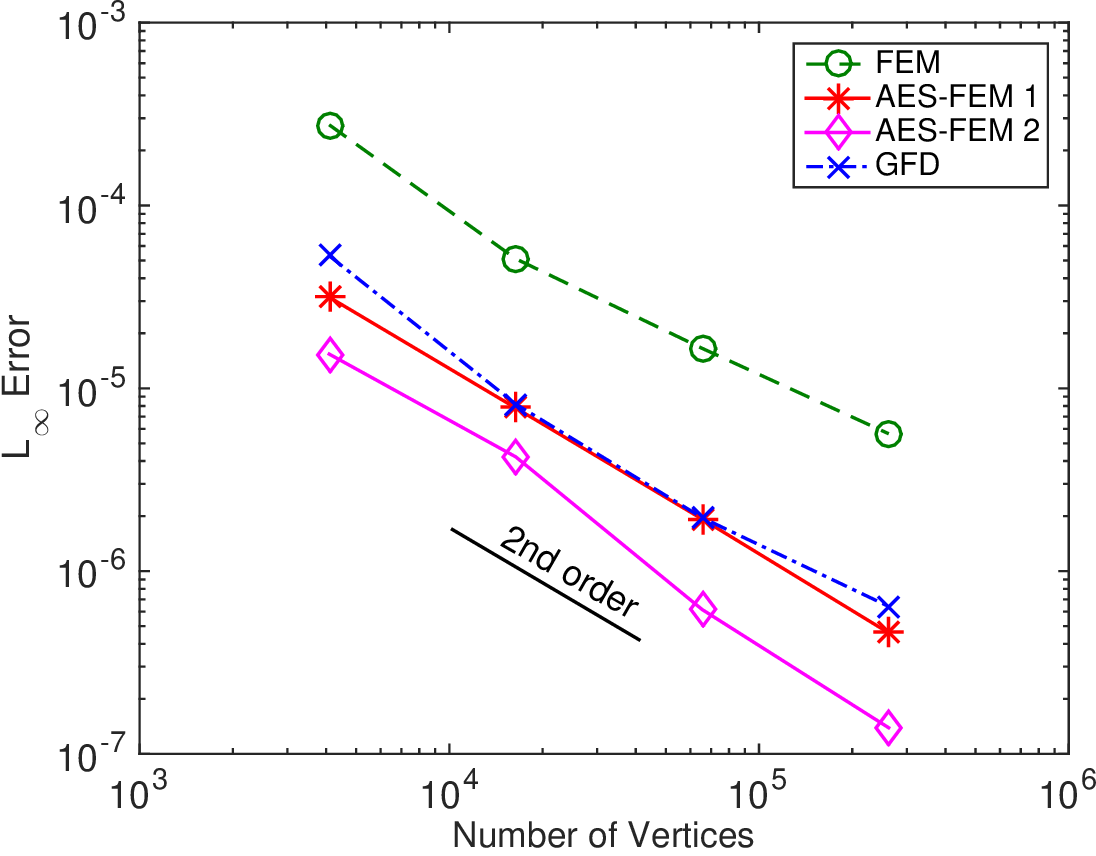}
\par\end{center}%
\end{minipage}

\raggedright{}\caption{The $L_{\infty}$ norm errors for the 2D Poisson equation on mesh
1 for $u_{2}$ (left) and $u_{3}$ (right).\label{fig:2D_Poisson_mesh1_u2=0000263}}
\end{figure}

\par\end{flushleft}

\begin{flushleft}
\begin{figure}
\begin{minipage}[t]{0.45\textwidth}%
\begin{center}
\includegraphics[width=1\textwidth]{2D_Pois_mesh2_u1_Inf}
\par\end{center}%
\end{minipage}\hfill{} %
\begin{minipage}[t]{0.45\textwidth}%
\begin{center}
\includegraphics[width=1\textwidth]{2D_Pois_mesh2_u1_L2}
\par\end{center}%
\end{minipage}

\raggedright{}\caption{The errors for 2D Poisson equation on mesh 2 for $u_{1}$. The errors
were computed using the $L_{\infty}$ norm (left) and the $L_{2}$
norm (right). \label{fig:2D_Poisson_mesh2_u1}}
\end{figure}

\par\end{flushleft}

\begin{flushleft}
\begin{figure}
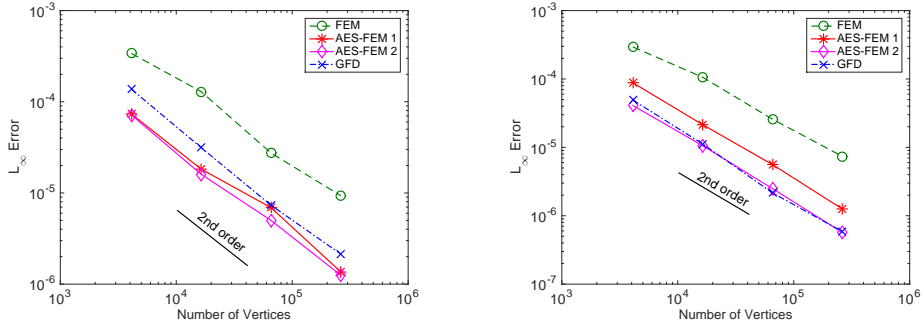

\begin{minipage}[t]{0.45\textwidth}%
\begin{center}
\includegraphics[width=1\textwidth]{2D_Pois_mesh2_u2_Inf}
\par\end{center}%
\end{minipage}\hfill{} %
\begin{minipage}[t]{0.45\textwidth}%
\begin{center}
\includegraphics[width=1\textwidth]{2D_Pois_mesh2_u3_Inf}
\par\end{center}%
\end{minipage}

\raggedright{}\caption{The $L_{\infty}$ norm errors for the 2D Poisson equation on mesh
2 for $u_{2}$ (left) and $u_{3}$ (right). \label{fig:2D_Poisson_mesh2_u2=0000263}}
\end{figure}

\par\end{flushleft}

\subsubsection{Convection-Diffusion Equation}

We consider the convection-diffusion equation with Dirichlet boundary
conditions on the unit square. That is, 
\begin{align}
-\nabla^{2}u+c\cdot\nabla u & =f\quad\text{in }\Omega,\\
u & =g\quad\text{on }\partial\Omega.
\end{align}
We take $\vec{c}=[1,1]^{T}$ for all of our tests and we consider
the same analytic solutions as for the Poisson equation. Again the
boundary conditions are obtained from the given analytic solutions.

The $L_{\infty}$ and $L_{2}$ norm errors obtained from the convection-diffusion
equation on mesh series 1 with $u_{1}$ are presented in Figure \ref{fig:2D_CD_mesh1_u1}.
The $L_{\infty}$ norm errors for $u_{2}$ and $u_{3}$ on mesh series
1 are in Figure \ref{fig:2D_CD_mesh1_u2=0000263}. For all three problems,
AES-FEM and GFD are both more accurate than linear FEM. For $u_{1}$,
GFD is the most accurate. For $u_{2}$, the most accurate method is
either AES-FEM 2 or GFD depending on the level of refinement. For
$u_{3}$, AES-FEM 2 is the most accurate.

On mesh series 2, AES-FEM 2 is the most accurate for $u_{1}$, as
can be seen in Figure~\ref{fig:2D_CD_mesh2_u1}. For $u_{2}$ AES-FEM
1 or AES-FEM 2 is the most accurate, and for $u_{3}$ GFD is the most
accurate; see Figure \ref{fig:2D_CD_mesh2_u2=0000263}. In all these
cases, AES-FEM is more accurate than FEM.

\begin{figure}
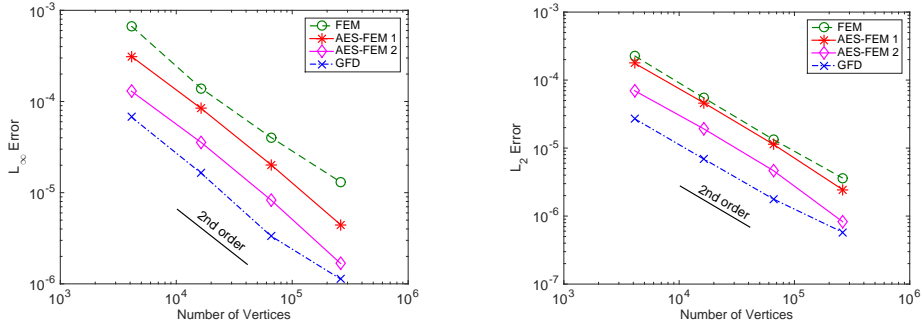

\begin{minipage}[t]{0.45\textwidth}%
\begin{center}
\includegraphics[width=1\textwidth]{2D_CD_mesh1_u1_Inf}
\par\end{center}%
\end{minipage}\hfill{} %
\begin{minipage}[t]{0.45\textwidth}%
\begin{center}
\includegraphics[width=1\textwidth]{2D_CD_mesh1_u1_L2}
\par\end{center}%
\end{minipage}

\raggedright{}\caption{The errors for 2D convection-diffusion equation on mesh 1 for $u_{1}$.
The errors were computed using the $L_{\infty}$ norm (left) and the
$L_{2}$ norm (right). \label{fig:2D_CD_mesh1_u1}}
\end{figure}

\begin{flushleft}
\begin{figure}
\begin{minipage}[t]{0.45\textwidth}%
\begin{center}
\includegraphics[width=1\textwidth]{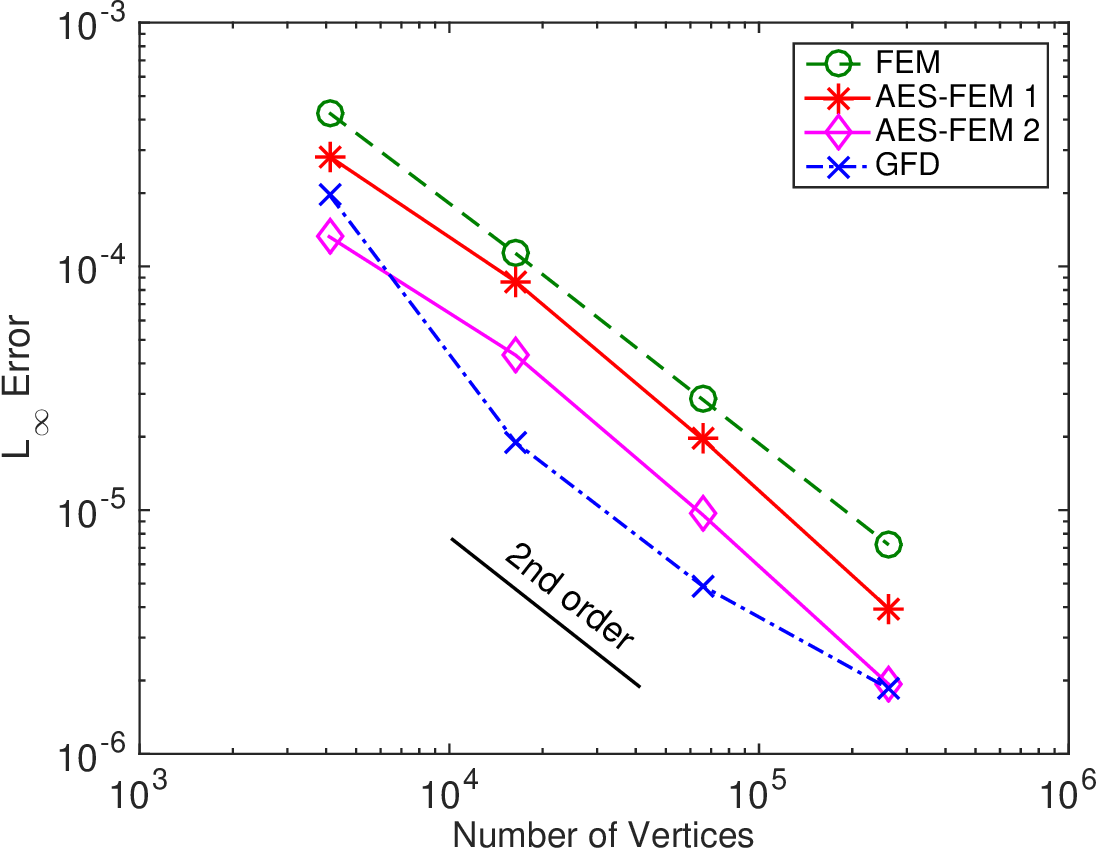}
\par\end{center}%
\end{minipage}\hfill{} %
\begin{minipage}[t]{0.45\textwidth}%
\begin{center}
\includegraphics[width=1\textwidth]{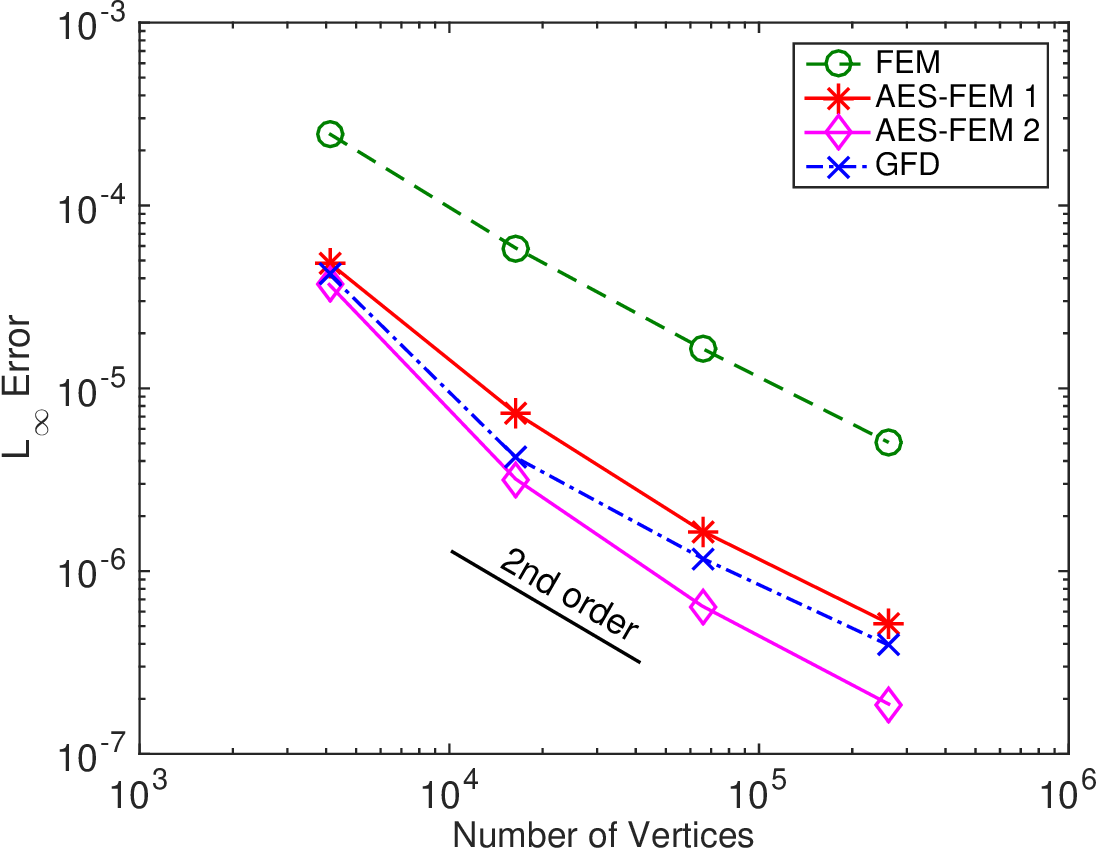}
\par\end{center}%
\end{minipage}

\raggedright{}\caption{The $L_{\infty}$ norm errors for the 2D convection-diffusion equation
on mesh 1 for $u_{2}$ (left) and $u_{3}$ (right).\label{fig:2D_CD_mesh1_u2=0000263}}
\end{figure}

\par\end{flushleft}

\begin{flushleft}
\begin{figure}
\begin{minipage}[t]{0.45\textwidth}%
\begin{center}
\includegraphics[width=1\textwidth]{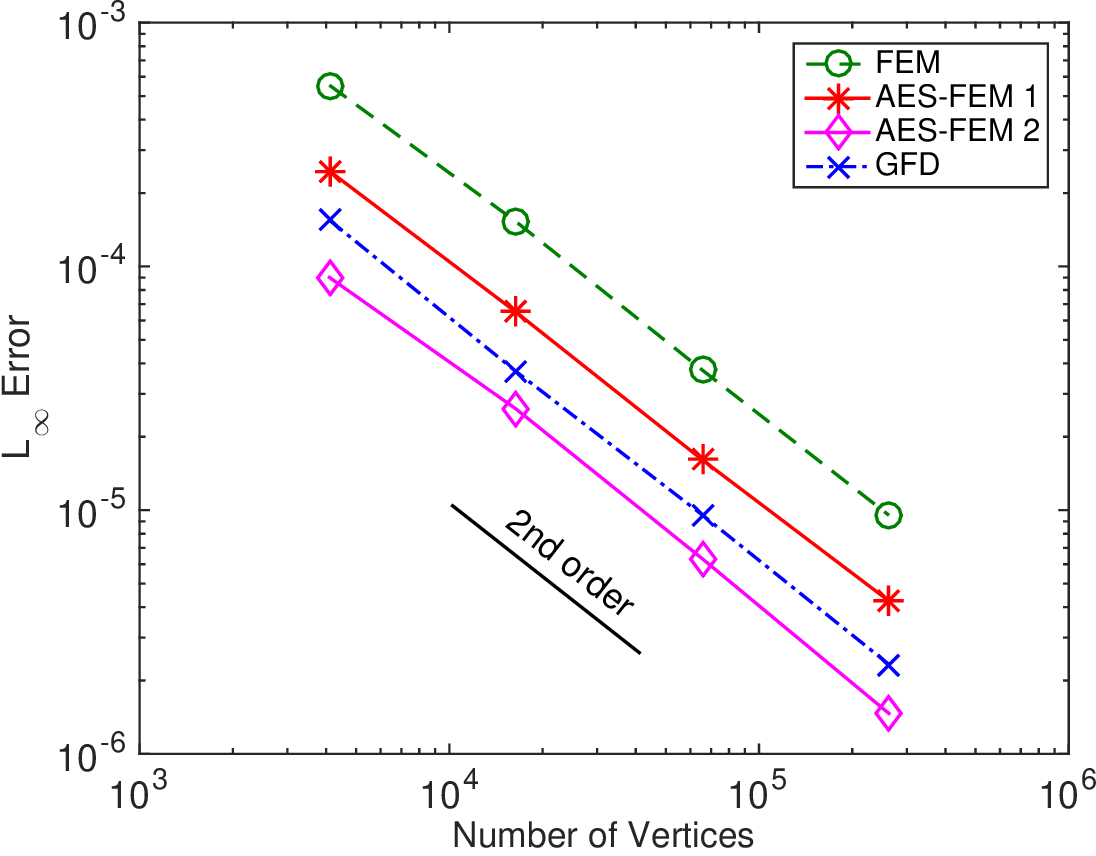}
\par\end{center}%
\end{minipage}\hfill{} %
\begin{minipage}[t]{0.45\textwidth}%
\begin{center}
\includegraphics[width=1\textwidth]{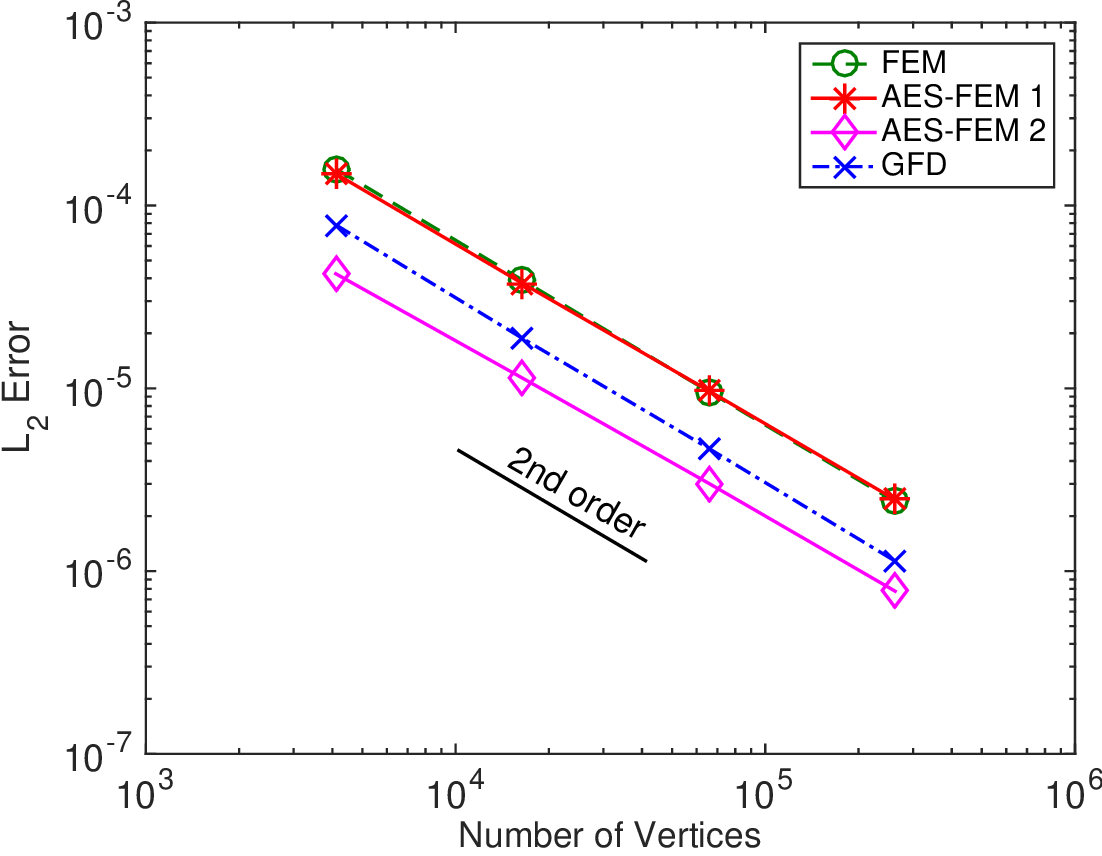}
\par\end{center}%
\end{minipage}

\raggedright{}\caption{The errors for 2D convection-diffusion equation on mesh 2 for $u_{1}$.
The errors were computed using the $L_{\infty}$ norm (left) and the
$L_{2}$ norm (right). \label{fig:2D_CD_mesh2_u1}}
\end{figure}

\par\end{flushleft}

\begin{flushleft}
\begin{figure}
\begin{minipage}[t]{0.45\textwidth}%
\begin{center}
\includegraphics[width=1\textwidth]{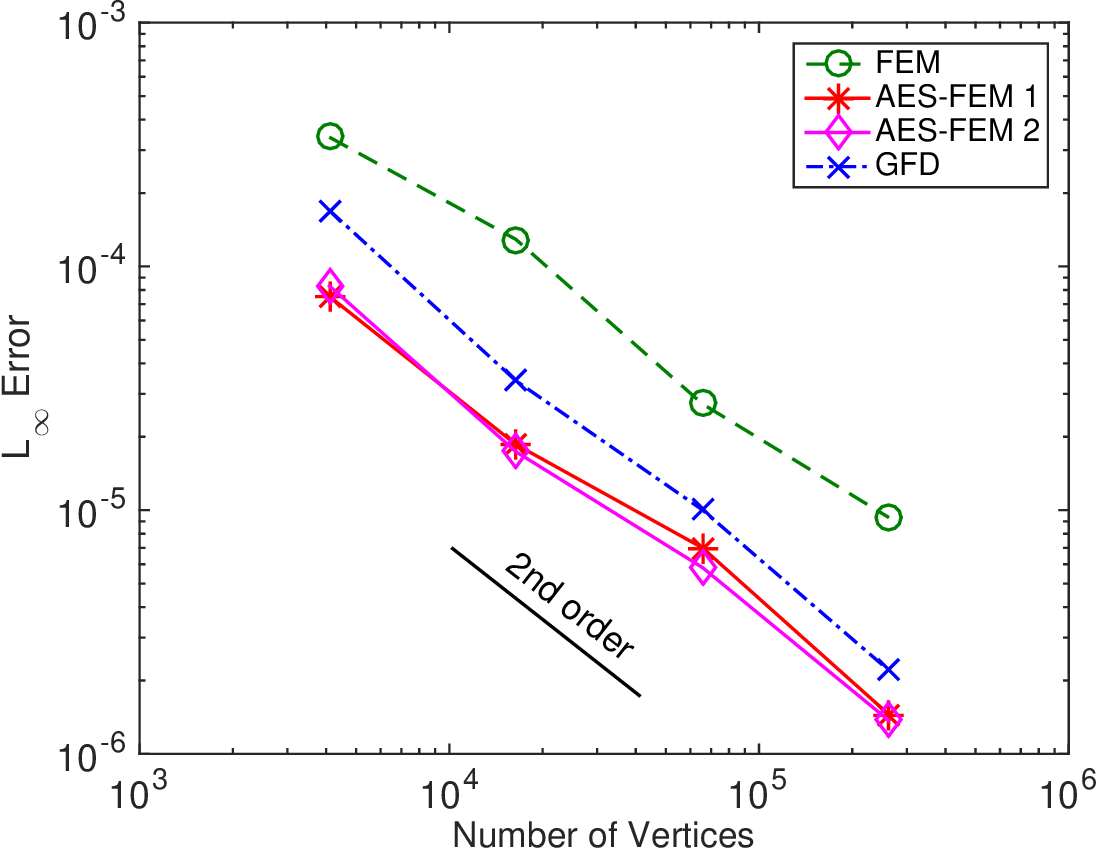}
\par\end{center}%
\end{minipage}\hfill{} %
\begin{minipage}[t]{0.45\textwidth}%
\begin{center}
\includegraphics[width=1\textwidth]{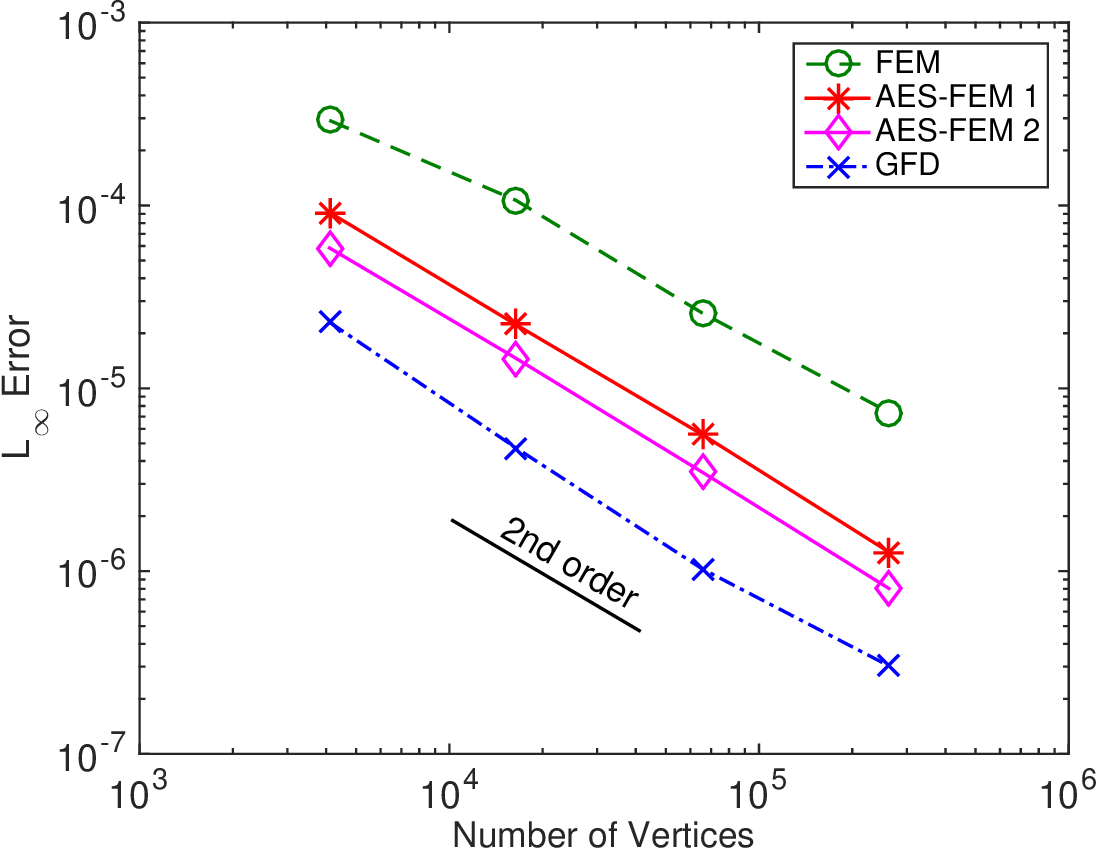}
\par\end{center}%
\end{minipage}

\raggedright{}\caption{The $L_{\infty}$ norm errors for the 2D convection-diffusion equation
on mesh 2 for $u_{2}$ (left) and $u_{3}$ (right). \label{fig:2D_CD_mesh2_u2=0000263}}
\end{figure}

\par\end{flushleft}

\subsubsection{Element-Quality Dependence Test \label{sub:2D_mesh_quality_test}}

We test how FEM, AES-FEM, and GFD perform on a series of progressively
worse meshes. We begin with the most refined mesh from mesh series
2. We select 6 of the 523,148 elements and incrementally move one
of their nodes towards the opposite edge so as to create flatter triangles.
We then solve the Poisson equation with the polynomial analytic solution
$u_{1}$ in (\ref{eq:polynomial_solution_2D}) and record the condition
numbers of the coefficient and stiffness matrices and the numbers
of iterations required for the solver to converge. Since the stiffness
matrix is the same for AES-FEM 1 and AES-FEM 2, the results are just
labeled as AES-FEM. We use the conjugate gradient method with incomplete
Cholesky preconditioner for FEM and we use GMRES with incomplete LU
preconditioner for AES-FEM and GFD. The tolerance for the solvers
is $10^{-8}$ and the drop tolerance for the preconditioners is $10^{-3}$.
As a measure of the mesh quality, we consider the cotangent of the
minimum angle in the mesh; as the minimum angle tends to zero, the
cotangent tends towards infinity. For very small angles, the cotangent
of the angle is approximately equal to the reciprocal of the angle.
We estimate the condition numbers using the MATLAB function $\mathtt{condest}$,
which computes a lower bound for the 1-norm condition number.

The worse the mesh quality, the higher the condition number of the
stiffness matrix resulting from FEM. In contrast, the condition numbers
of the GFD coefficient matrix and stiffness matrix from AES-FEM remain
almost constant. As the condition number for FEM rises, so does the
number of iterations required for the solver to converge, from 102
to 128. The numbers of iterations required to solve the equation for
AES-FEM and GFD remain constant, at 73 and 74, respectively. We show
the results for 6 meshes. Preconditioned conjugate gradient stagnates
when trying to solve the FEM linear system from the 7th mesh, where
the minimum angle is approximately $9.1\times10^{-5}$ degrees. Solving
the AES-FEM and GFD linear systems from the 7th mesh requires the
same numbers of iterations as the other meshes, 73 and 74 respectively.
See Figure \ref{fig:2D_poor_quality_meshes} for a comparison of the
condition numbers and the numbers of iterations. 

The errors for both AES-FEM 1 and AES-FEM 2 rose slightly between
the 3rd and the 4th mesh; the errors were then constant for the rest
of the meshes. The errors for FEM remained constant and the errors
for GFD remained nearly constant over the 6 meshes. AES-FEM 1, AES-FEM
2 and GFD converge on the 7th mesh in the series with the same errors
as on Mesh 6, whereas for FEM the solver stagnates. See Table \ref{tab:Errors-2D-Bad_meshes}
for specific errors.

\begin{table}
\begin{centering}
\caption{Errors in $L_{2}$ norm for FEM, AES-FEM 1, AES-FEM 2, and GFD for
$u_{1}$ on a series of meshes with progressively worse mesh element
quality.\label{tab:Errors-2D-Bad_meshes}}

\par\end{centering}

\centering{}%
\begin{tabular}{|c|c|c|c|c|}
\hline 
 & FEM & AES-FEM 1 & AES-FEM 2 & GFD\tabularnewline
\hline 
\hline 
Mesh 1 & $2.42\times10^{-6}$ & $2.47\times10^{-6}$ & $7.82\times10^{-7}$ & $1.11\times10^{-6}$\tabularnewline
\hline 
Mesh 2 & $2.42\times10^{-6}$ & $2.47\times10^{-6}$ & $7.83\times10^{-7}$ & $1.10\times10^{-6}$\tabularnewline
\hline 
Mesh 3 & $2.42\times10^{-6}$ & $2.47\times10^{-6}$ & $7.82\times10^{-7}$ & $1.10\times10^{-6}$\tabularnewline
\hline 
Mesh 4 & $2.42\times10^{-6}$ & $2.81\times10^{-6}$ & $1.19\times10^{-6}$ & $1.10\times10^{-6}$\tabularnewline
\hline 
Mesh 5 & $2.42\times10^{-6}$ & $2.81\times10^{-6}$ & $1.19\times10^{-6}$ & $1.10\times10^{-6}$\tabularnewline
\hline 
Mesh 6 & $2.42\times10^{-6}$ & $2.81\times10^{-6}$ & $1.19\times10^{-6}$ & $1.10\times10^{-6}$\tabularnewline
\hline 
\end{tabular}
\end{table}

\begin{figure}
\begin{minipage}[t]{0.45\textwidth}%
\begin{center}
\includegraphics[width=1\textwidth]{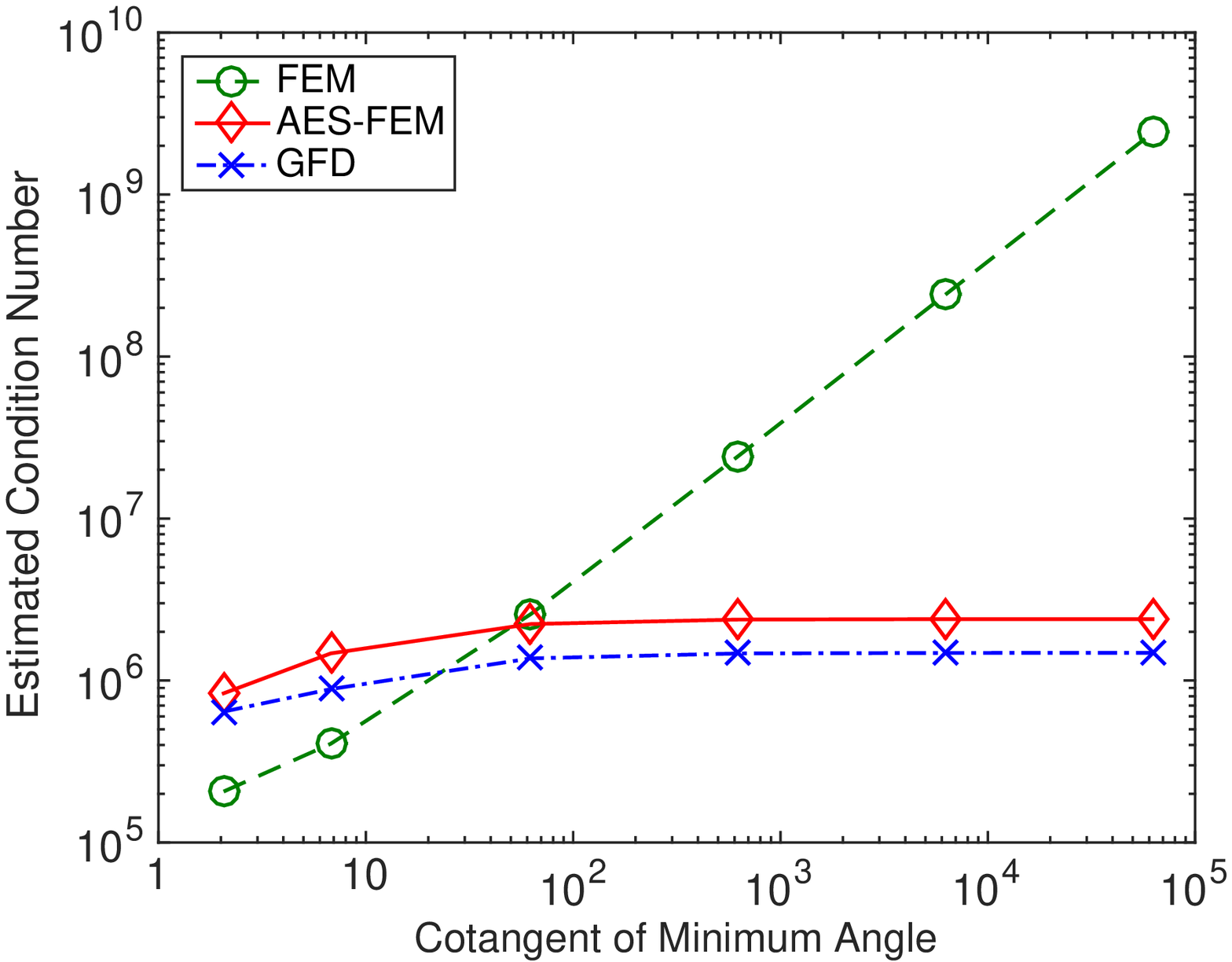}
\par\end{center}%
\end{minipage}\hfill{} %
\begin{minipage}[t]{0.45\textwidth}%
\begin{center}
\includegraphics[width=1\textwidth]{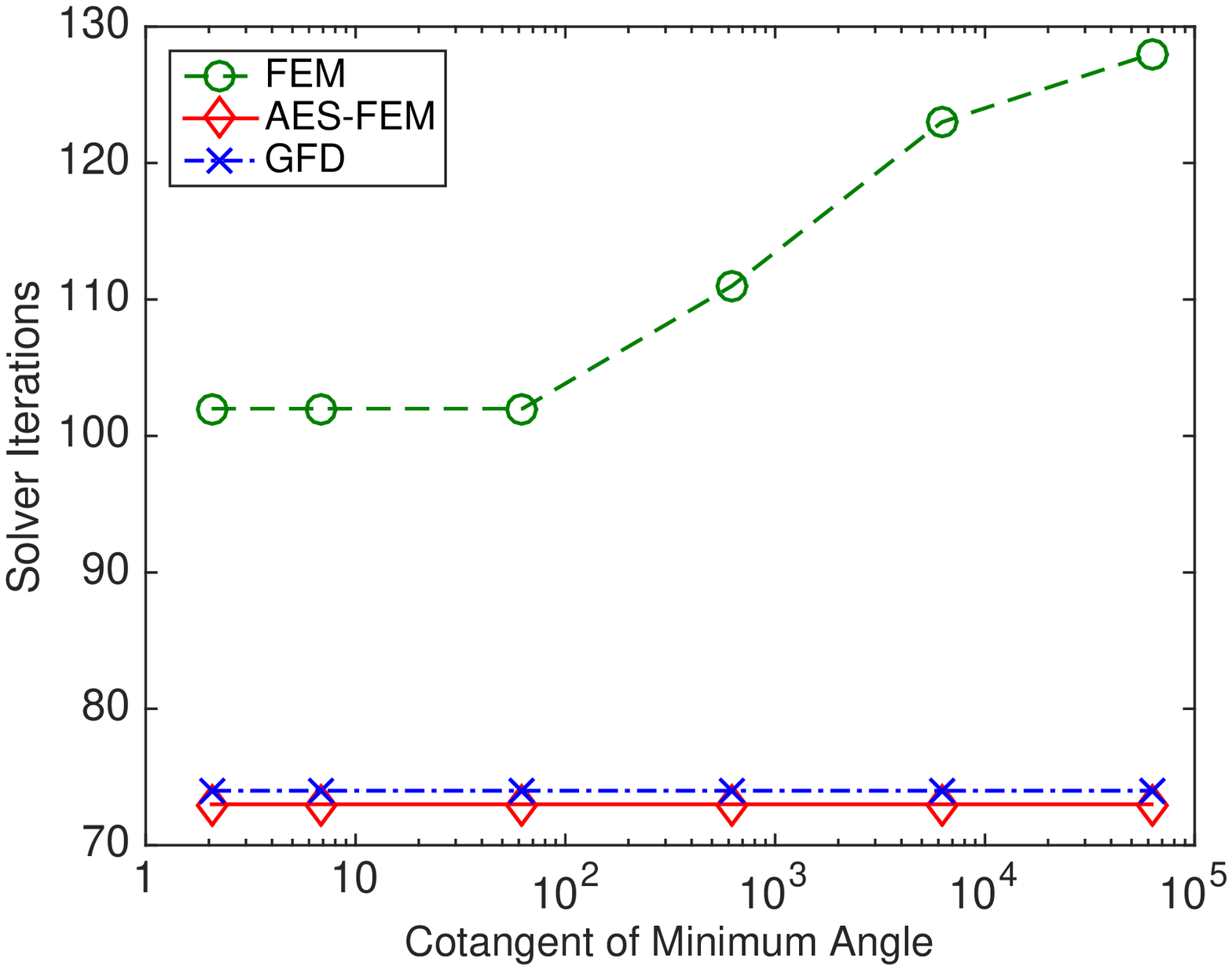}
\par\end{center}%
\end{minipage}

\raggedright{}\caption{The condition numbers of the stiffness matrices for FEM and Adaptive
Extended Stencil (AES)-FEM and the coefficient matrix for generalized
finite difference (GFD) (left) and the numbers of solver iterations
(right). Solvers used are preconditioned conjugate gradient (PCG)
for FEM and preconditioned generalized minimal residual (GMRES) for
AES-FEM and GFD.\label{fig:2D_poor_quality_meshes}}
\end{figure}

\subsubsection{Efficiency}

We compare the runtimes for the four methods: AES-FEM 1, AES-FEM 2,
FEM, and GFD. We consider the convection-diffusion equation on the
most refined mesh of series 2 with the polynomial analytic solution
$u_{1}$ for this runtime experiment. We decompose the total time
into 4 subcategories: \textit{Initialization,} which includes the
time to load the mesh and the time to assign the boundary conditions
and problem values; \textit{Assembly,} which includes the time to
build the stiffness matrix and load vector; \textit{Preconditioner,}
which is the time it takes to construction the matrix preconditioner
using incomplete LU factorization with a drop tolerance of $10^{-3}$;
and \textit{Solver,} which is the amount of time for solving the preconditioned
system using GMRES with a tolerance of $10^{-8}$. See Figure \ref{fig:time_2D_CD}
for the comparison. The initialization time is minuscule compared
to the other categories and is not visible in the figure. FEM requires
76 iterations of GMRES to converge, AES-FEM 1 and AES-FEM 2 both require
74 iterations, and GFD requires 75 iterations.

One can see that FEM has an advantage when it comes to efficiency
on the same mesh. In terms of total runtime, it is about 2.1 times
faster than AES-FEM 1, 2.1 times faster than AES-FEM 2, and 1.8 times
faster than GFD. In terms of assembly time, FEM is about 6.1 times
faster than AES-FEM 1, 7.7 times faster than AES-FEM 2, and 2.2 times
faster than GFD. Comparing the assembly time of AES-FEM and GFD, we
see that GFD is approximately 3.2 and 3.5 times faster than AES-FEM
1 and AES-FEM 2, respectively. 

In 2D, assembling the load vector using FEM basis functions (AES-FEM
1) saves some time compared to using GLP basis functions (AES-FEM
2). AES-FEM 1 offers a savings of approximately 1.1 seconds or, in
other words, a 10.3\% reduction of assembly time and a 3.5\% reduction
of total time compared to AES-FEM 2. We will see in the next section
that the efficiency of these two methods varies more in 3D. 

\begin{figure}
\begin{centering}
\includegraphics[scale=0.6]{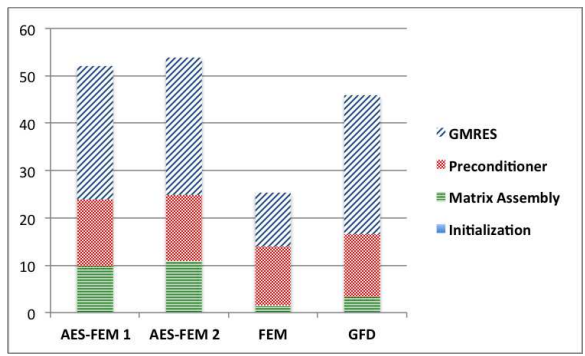}
\par\end{centering}

\caption{Runtimes for a 2D convection-diffusion equation on the most refined
mesh in series 2.\label{fig:time_2D_CD}}
\end{figure}

However, in terms of error versus runtime, AES-FEM is competitive
with, and often is more efficient than, the classical FEM with linear
basis functions. In Figure~\ref{fig:time_vs_error_2d}, we compare
the $L_{\infty}$ norm errors versus runtimes for the four methods
on mesh series 2 for the Poisson equation and the convection-diffusion
equation with the analytic solution equal to $u_{2}$. For the Poisson
equation, all four methods are very similar, with AES-FEM 1 being
slightly more efficient for finer meshes. For the convection-diffusion
equation, AES-FEM 1 and AES-FEM 2 are approximately the same in terms
of efficiency and are the most efficient. GFD is also more efficient
than FEM.

\begin{figure}
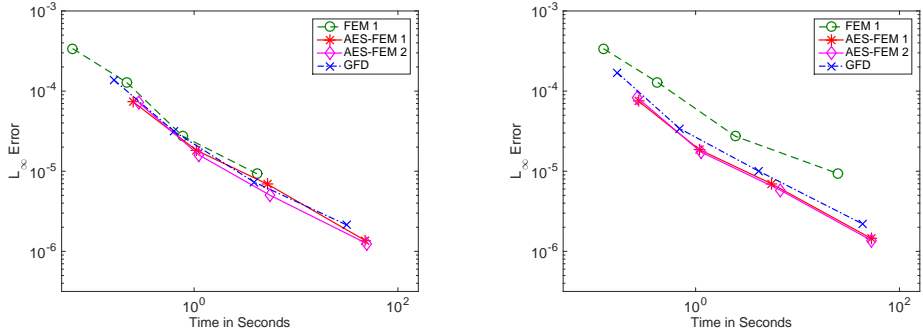

\begin{minipage}[t]{0.45\textwidth}%
\begin{center}
\includegraphics[width=1\textwidth]{error_vs_time_pois_mesh2_u2_inf_2d}
\par\end{center}%
\end{minipage}\hfill{} %
\begin{minipage}[t]{0.45\textwidth}%
\begin{center}
\includegraphics[width=1\textwidth]{error_vs_time_cd_mesh2_u2_inf_2d}
\par\end{center}%
\end{minipage}

\raggedright{}\caption{$L_{\infty}$ norm errors versus runtimes for a 2D Poisson equation
(left) and convection-diffusion equation (right) on mesh series 2.
Lower is better.\label{fig:time_vs_error_2d} }
\end{figure}

\subsection{3D Results}

In this section, we present the results from the 3D experiments. We
consider the Poisson equation and the convection-diffusion equation
in 3D. We test three problems for each equation on two different series
of meshes, each with four levels of refinement. The first series of
meshes (referred to collectively as ``mesh series 1'') is created
by placing nodes on a regular grid and using MATLAB's Delaunay triangularization
to create the elements. The meshes in series 1 range from $8\times8\times8$
nodes to $64\times64\times64$ nodes. The minimum dihedral angle in
the most refined mesh of series 1 is 35.2 degrees and the maximum
dihedral angle is 125.2 degrees. The maximum aspect ratio is 4.9,
where the aspect ratio of a tetrahedron is defined as the ratio of
the longest edge length to the smallest height. The second series
of meshes (referred to collectively as ``mesh series 2'') is created
using TetGen \cite{si2015tetgen}. The number of nodes in mesh series
2 for each level of refinement is 509, 4,080, 32,660, and 261,393,
which is approximately the same as the meshes in series 1. The minimum
dihedral angle of the most refined mesh in series 2 is 6.7 degrees
and the largest dihedral angle is 165.5 degrees. The largest aspect
ratio is 15.2.

\subsubsection{Poisson Equation}

We first consider the Poisson equation with Dirichlet boundary conditions
on the unit cube. That is,
\begin{align}
-\nabla^{2}u & =f\quad\text{in }\Omega,\\
u & =g\quad\text{on }\partial\Omega.
\end{align}
where $\Omega=[0,1]^{3}$. We consider three different analytic solutions
listed below.
\begin{align}
u_{1} & =64x\left(1-x\right)y\left(1-y\right)z\left(1-z\right),\label{eq:polynomial_solution_3D}\\
u_{2} & =\cos(\pi x)\cos(\pi y)\cos(\pi z),\label{eq:trigonometric_solution_3D}\\
u_{3} & =\frac{1}{\sinh\pi\cosh\pi\cosh\pi}\sinh(\pi x)\cosh(\pi y)\cosh(\pi z).\label{eq:hyerbolic_solution_3D}
\end{align}
The Dirichlet boundary conditions are derived from the analytic solutions.
They are homogeneous for $u_{1}$ and non-homogeneous for $u_{2}$
and $u_{3}$.

The $L_{\infty}$ and $L_{2}$ norm errors for the Poisson equation
on mesh series 1 for $u_{1}$ can be seen in Figure \ref{fig:3D_Poisson_mesh1_u1}
and the $L_{\infty}$ norm errors for $u_{2}$ and $u_{3}$ can be
seen in Figure \ref{fig:3D_Poisson_mesh1_u2=0000263}. GFD is the
most accurate for $u_{1}$ and $u_{2}$. AES-FEM 2 is the most accurate
for $u_{3}$.

For mesh series 2, AES-FEM 2 is the most accurate and is close to
an order of magnitude more accurate than FEM. See Figure \ref{fig:3D_Poisson_mesh2_u1}
for the $L_{\infty}$ and $L_{2}$ norm errors for $u_{1}$ and Figure
\ref{fig:3D_Poisson_mesh2_u2=0000263} for the $L_{\infty}$ norm
errors for $u_{2}$ and $u_{3}$. 

\begin{flushleft}
\begin{figure}
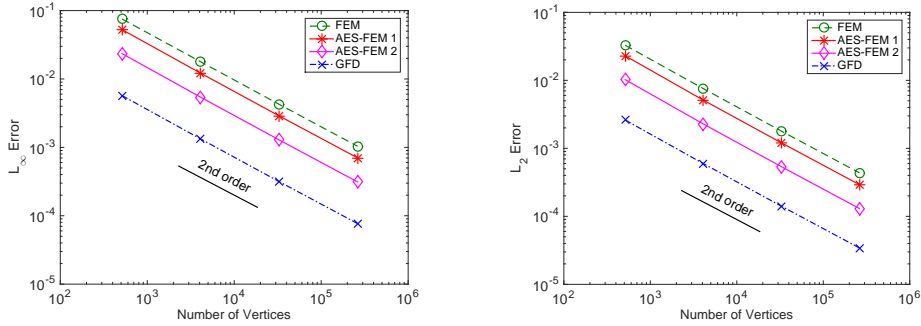

\begin{minipage}[t]{0.45\textwidth}%
\begin{center}
\includegraphics[width=1\textwidth]{3D_Pois_mesh1_u1_Inf}
\par\end{center}%
\end{minipage}\hfill{} %
\begin{minipage}[t]{0.45\textwidth}%
\begin{center}
\includegraphics[width=1\textwidth]{3D_Pois_mesh1_u1_L2}
\par\end{center}%
\end{minipage}

\raggedright{}\caption{The errors for 3D Poisson equation on mesh 1 with $u_{1}$. The errors
were computed using the $L_{\infty}$ norm (left) and the $L_{2}$
norm (right). \label{fig:3D_Poisson_mesh1_u1}}
\end{figure}

\par\end{flushleft}

\begin{flushleft}
\begin{figure}
\begin{minipage}[t]{0.45\textwidth}%
\begin{center}
\includegraphics[width=1\textwidth]{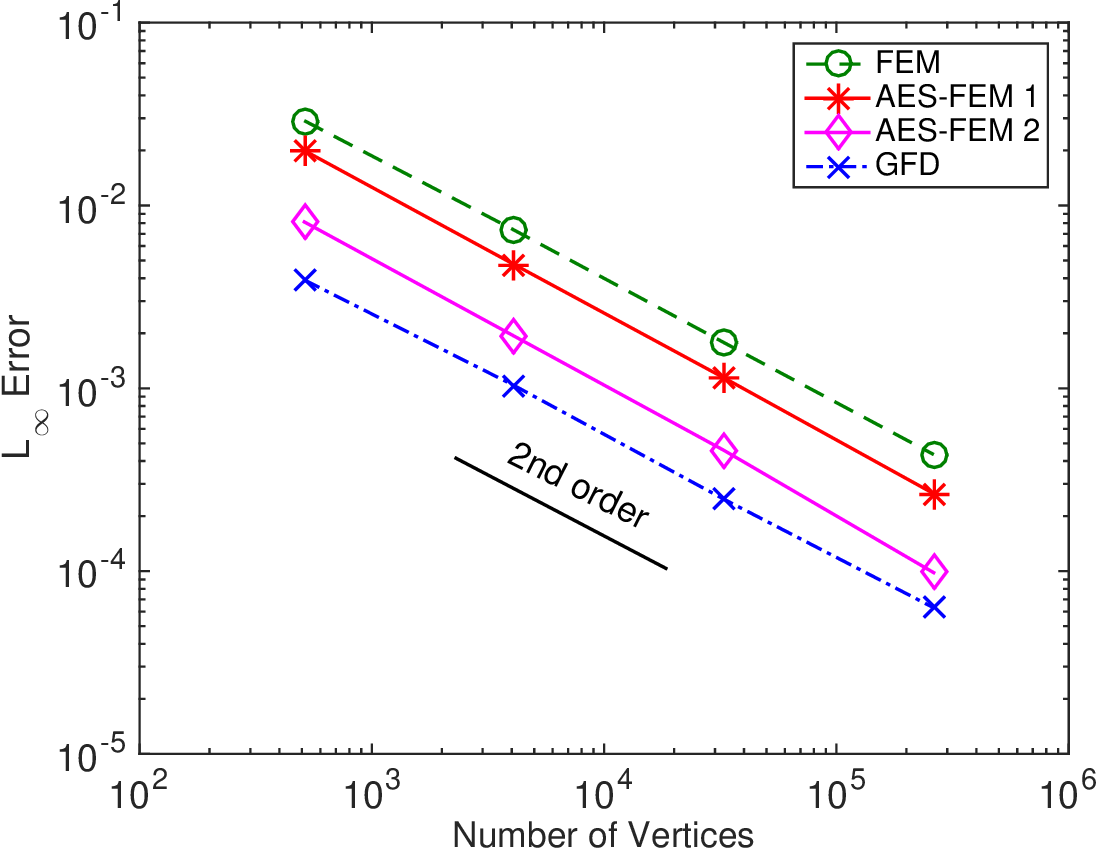}
\par\end{center}%
\end{minipage}\hfill{} %
\begin{minipage}[t]{0.45\textwidth}%
\begin{center}
\includegraphics[width=1\textwidth]{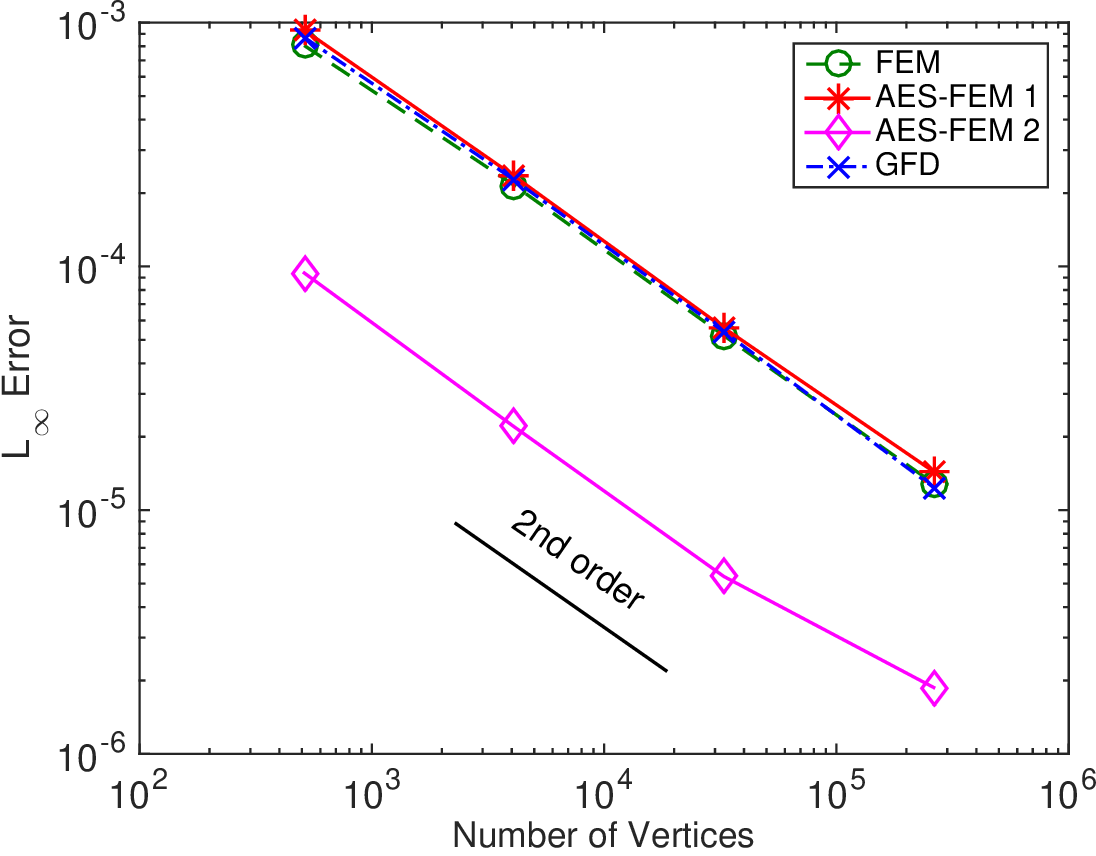}
\par\end{center}%
\end{minipage}

\raggedright{}\caption{The $L_{\infty}$ norm errors for the 3D Poisson equation on mesh
1 for $u_{2}$ (left) and $u_{3}$ (right).\label{fig:3D_Poisson_mesh1_u2=0000263}}
\end{figure}

\par\end{flushleft}

\begin{flushleft}
\begin{figure}
\begin{minipage}[t]{0.45\textwidth}%
\begin{center}
\includegraphics[width=1\textwidth]{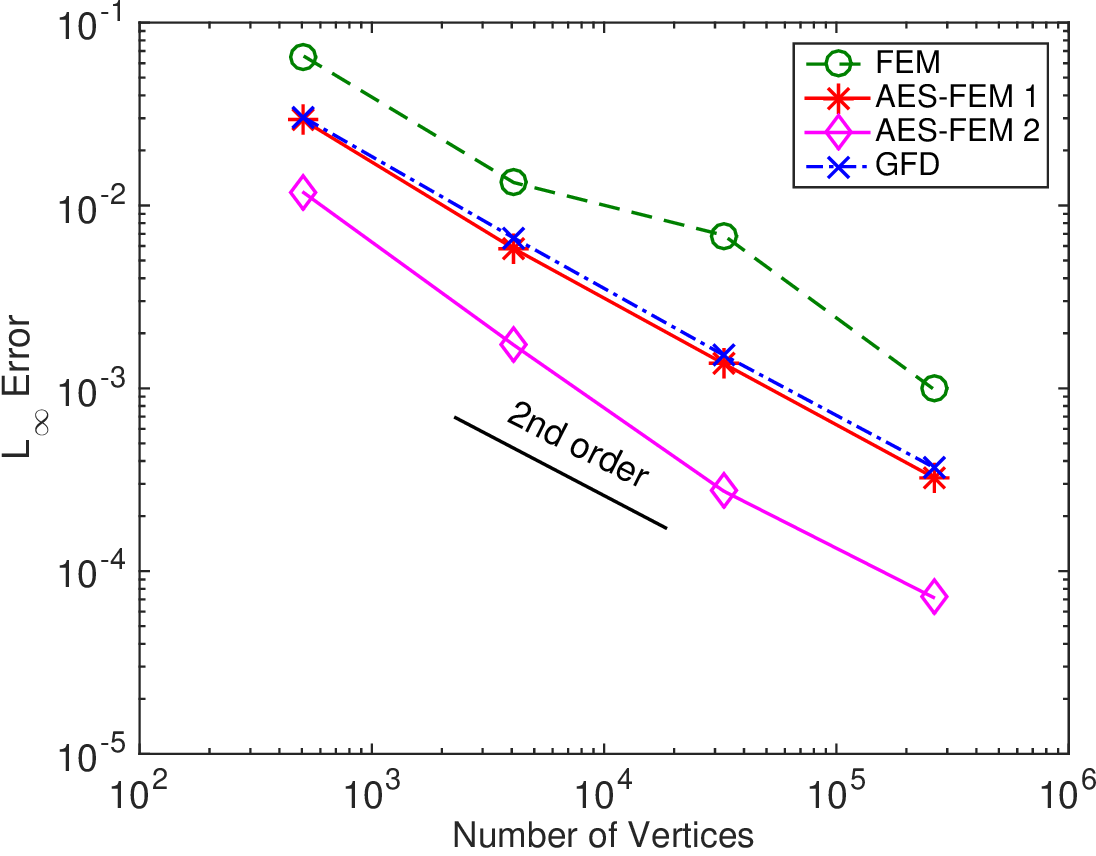}
\par\end{center}%
\end{minipage}\hfill{} %
\begin{minipage}[t]{0.45\textwidth}%
\begin{center}
\includegraphics[width=1\textwidth]{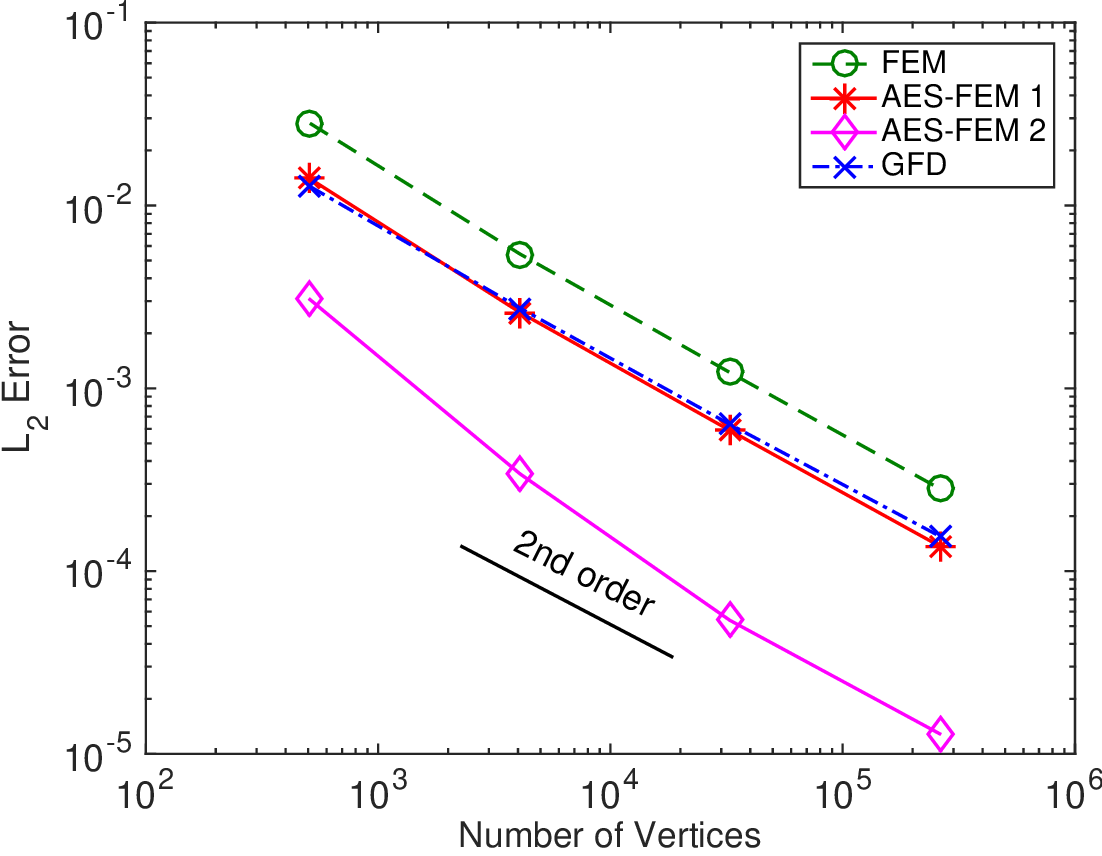}
\par\end{center}%
\end{minipage}

\raggedright{}\caption{The errors for 3D Poisson equation on mesh 2 for $u_{1}$. The errors
were computed using the $L_{\infty}$ norm (left) and the $L_{2}$
norm (right). \label{fig:3D_Poisson_mesh2_u1}}
\end{figure}

\par\end{flushleft}

\begin{flushleft}
\begin{figure}
\begin{minipage}[t]{0.45\textwidth}%
\begin{center}
\includegraphics[width=1\textwidth]{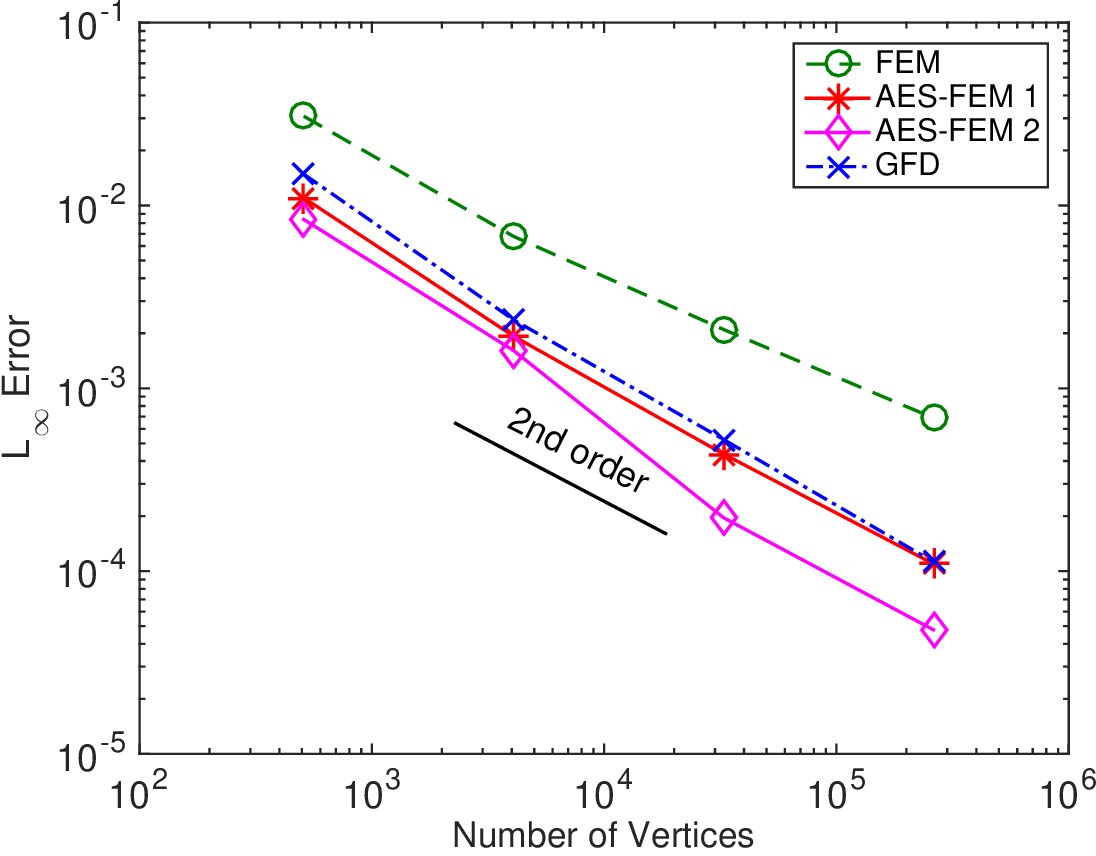}
\par\end{center}%
\end{minipage}\hfill{} %
\begin{minipage}[t]{0.45\textwidth}%
\begin{center}
\includegraphics[width=1\textwidth]{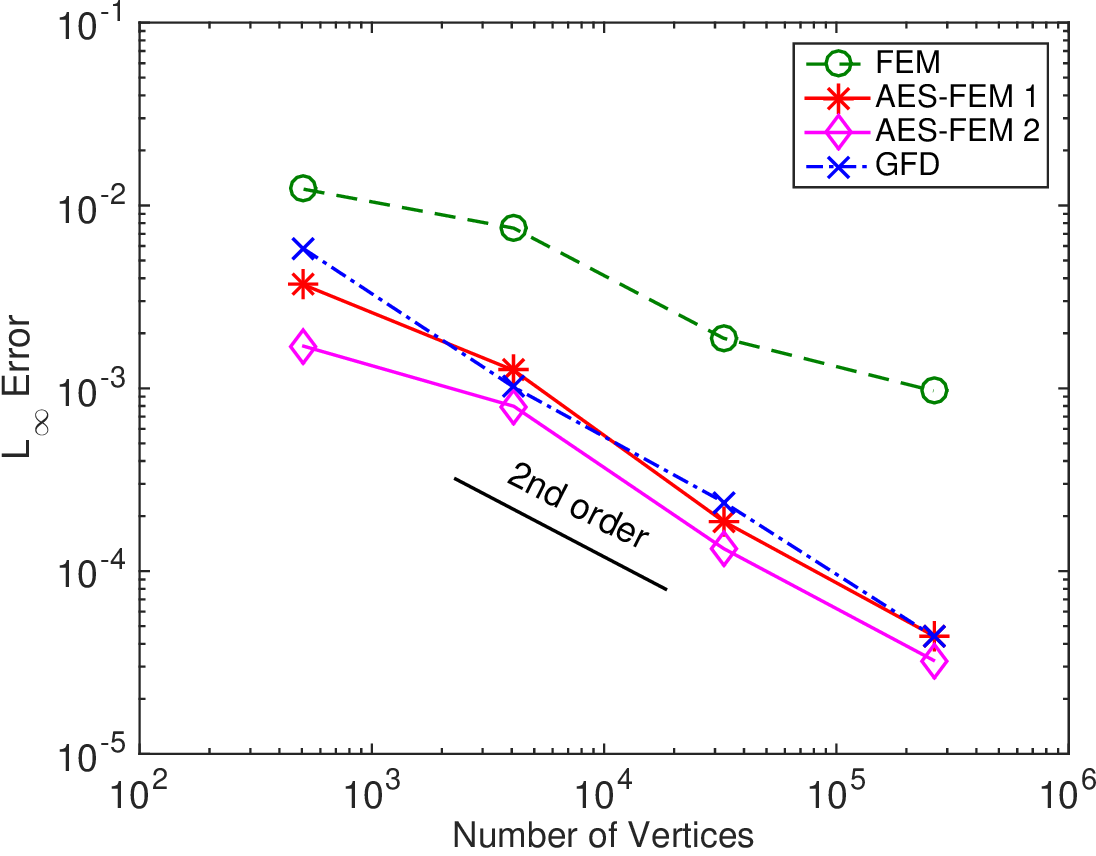}
\par\end{center}%
\end{minipage}

\raggedright{}\caption{The $L_{\infty}$ norm errors for the 3D Poisson equation on mesh
2 for $u_{2}$ (left) and $u_{3}$ (right). \label{fig:3D_Poisson_mesh2_u2=0000263}}
\end{figure}

\par\end{flushleft}

\subsubsection{Convection-Diffusion Equation}

We consider the convection-diffusion equation with Dirichlet boundary
conditions on the unit cube, $\Omega=[0,1]^{3}$. 
\begin{align}
-\nabla^{2}u+c\cdot\nabla u & =f\quad\text{in }\Omega,\\
u & =g\quad\text{on }\partial\Omega.
\end{align}
We take $\vec{c}=[1,1,1]^{T}$ and we consider the same analytic solutions
as in the previous section. Again, the Dirichlet boundary conditions
are derived from the analytic solutions $u_{1}$, $u_{2}$, and $u_{3}$. 

The $L_{\infty}$ and $L_{2}$ norm errors for the 3D convection-diffusion
equation on mesh series 1 for $u_{1}$ can be seen in Figure \ref{fig:3D_CD_mesh1_u1},
see Figure \ref{fig:3D_CD_mesh1_u2=0000263} for the $L_{\infty}$
norm errors on mesh series 1 for $u_{2}$ and $u_{3}$. As with the
Poisson equation, GFD is the most accurate for $u_{1}$ and $u_{2}$.
For $u_{3}$, either GFD or AES-FEM 2 is the most accurate depending
on the level of refinement. 

On mesh series 2, AES-FEM 2 is the most accurate for $u_{1}$ and
$u_{2}$, as can be seen in Figure~\ref{fig:3D_CD_mesh2_u1} and
the left panel of Figure~\ref{fig:3D_CD_mesh2_u2=0000263}. For $u_{3}$,
either GFD or AES-FEM 2 is the most accurate, as can be seen in the
right panel of Figure~\ref{fig:3D_CD_mesh2_u2=0000263}. Similar
to 2D results, AES-FEM is more accurate than FEM in all these cases.

\begin{figure}
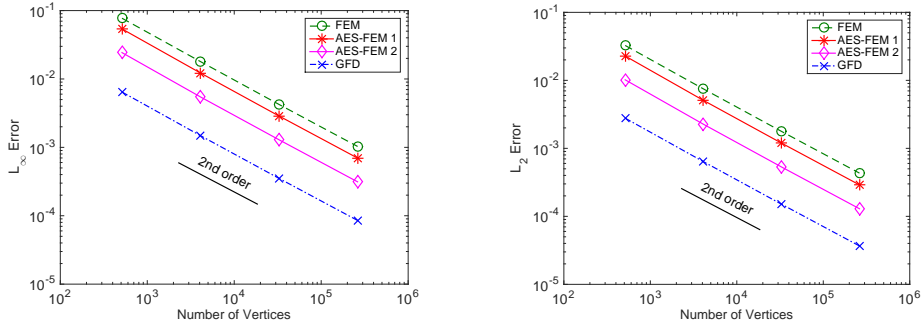

\begin{minipage}[t]{0.45\textwidth}%
\begin{center}
\includegraphics[width=1\textwidth]{3D_CD_mesh1_u1_Inf}
\par\end{center}%
\end{minipage}\hfill{} %
\begin{minipage}[t]{0.45\textwidth}%
\begin{center}
\includegraphics[width=1\textwidth]{3D_CD_mesh1_u1_L2}
\par\end{center}%
\end{minipage}

\raggedright{}\caption{The errors for 3D convection-diffusion equation on mesh 1 for $u_{1}$.
The errors were computed using the $L_{\infty}$ norm (left) and the
$L_{2}$ norm (right). \label{fig:3D_CD_mesh1_u1}}
\end{figure}

\begin{flushleft}
\begin{figure}
\begin{minipage}[t]{0.45\textwidth}%
\begin{center}
\includegraphics[width=1\textwidth]{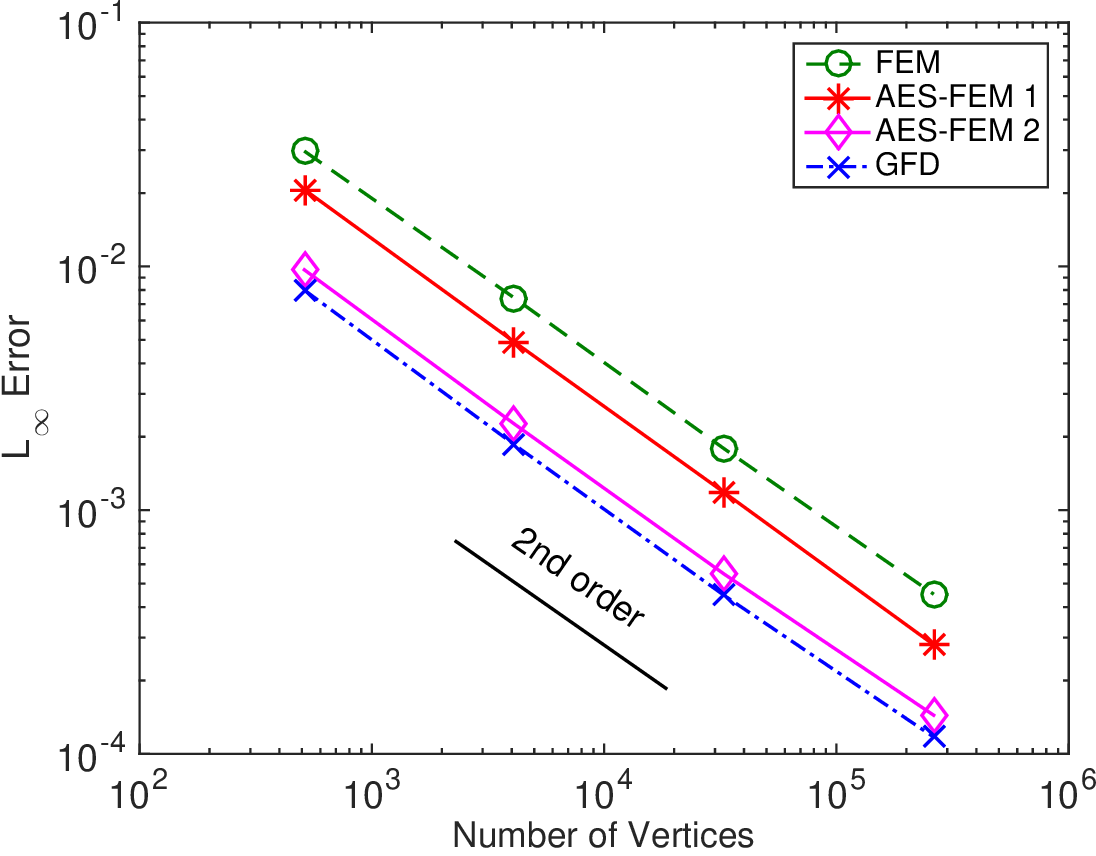}
\par\end{center}%
\end{minipage}\hfill{} %
\begin{minipage}[t]{0.45\textwidth}%
\begin{center}
\includegraphics[width=1\textwidth]{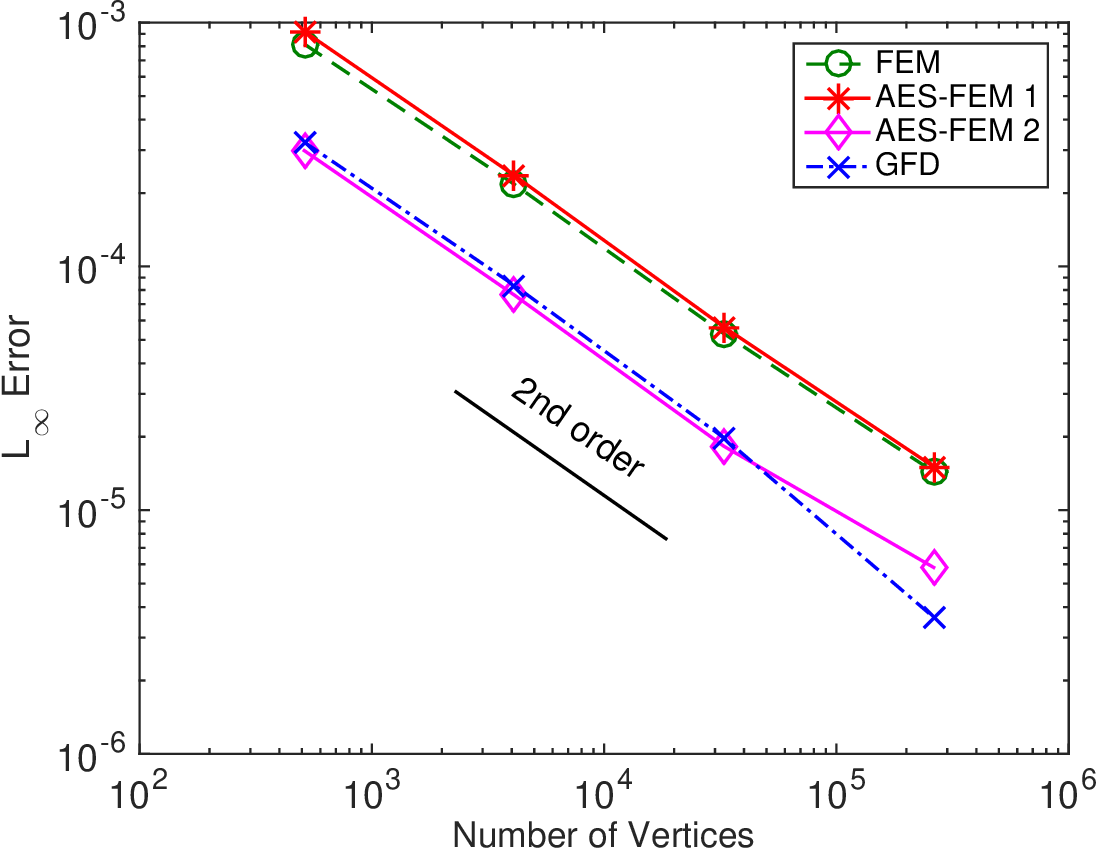}
\par\end{center}%
\end{minipage}

\raggedright{}\caption{The $L_{\infty}$ norm errors for the 3D convection-diffusion equation
on mesh 1 for $u_{2}$ (left) and $u_{3}$ (right).\label{fig:3D_CD_mesh1_u2=0000263}}
\end{figure}

\par\end{flushleft}

\begin{flushleft}
\begin{figure}
\begin{minipage}[t]{0.45\textwidth}%
\begin{center}
\includegraphics[width=1\textwidth]{3D_CD_mesh2_u1_Inf}
\par\end{center}%
\end{minipage}\hfill{} %
\begin{minipage}[t]{0.45\textwidth}%
\begin{center}
\includegraphics[width=1\textwidth]{3D_CD_mesh2_u1_L2}
\par\end{center}%
\end{minipage}

\raggedright{}\caption{The errors for 3D convection-diffusion equation on mesh 2 for $u_{1}$.
The errors were computed using the $L_{\infty}$ norm (left) and the
$L_{2}$ norm (right). \label{fig:3D_CD_mesh2_u1}}
\end{figure}

\par\end{flushleft}

\begin{flushleft}
\begin{figure}
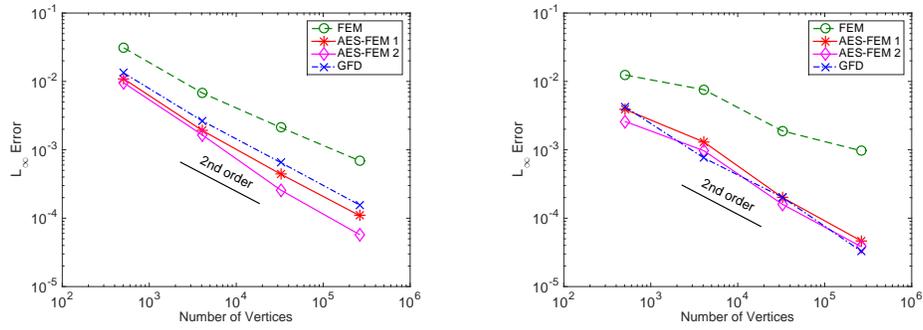

\begin{minipage}[t]{0.45\textwidth}%
\begin{center}
\includegraphics[width=1\textwidth]{3D_CD_mesh2_u2_Inf}
\par\end{center}%
\end{minipage}\hfill{} %
\begin{minipage}[t]{0.45\textwidth}%
\begin{center}
\includegraphics[width=1\textwidth]{3D_CD_mesh2_u3_Inf}
\par\end{center}%
\end{minipage}

\raggedright{}\caption{The $L_{\infty}$ errors for the 3D convection-diffusion equation
on mesh 2 for $u_{2}$ (left) and $u_{3}$ (right). \label{fig:3D_CD_mesh2_u2=0000263}}
\end{figure}

\par\end{flushleft}

\subsubsection{Element-Quality Dependence Test \label{sub:3D_mesh_quality_test}}

We test how FEM, AES-FEM, and GFD perform on a series of meshes with
progressively worse element shape quality. We begin with the most
refined mesh from mesh series 1. We select 69 out of the 1,500,282
elements and incrementally move one of their nodes towards the opposite
side so as to create sliver tetrahedra. We then solve the Poisson
equation with the polynomial analytic solution $u_{1}$ in (\ref{eq:polynomial_solution_3D})
and record the condition numbers of the coefficient and stiffness
matrices and the numbers of iterations required for the solver to
converge. We use the conjugate gradient method with Gauss-Seidel preconditioner
for FEM and we use GMRES with Gauss-Seidel preconditioner for AES-FEM
and GFD. We use a tolerance of $10^{-5}$ for both solvers. As a measure
of the mesh quality, we consider the cotangent of the minimum dihedral
angle in the mesh; as the minimum angle tends to zero, the cotangent
tends towards infinity. For very small angles, the cotangent of the
angle is approximately equal to the reciprocal of the angle. We estimate
the condition numbers using the MATLAB function $\mathtt{condest}$,
which computes a lower bound for the 1-norm condition number.

The worse the mesh quality, the higher the condition number of the
stiffness matrix resulting from FEM. The condition numbers of the
stiffness matrix from AES-FEM and the coefficient matrix from GFD
remain almost constant. As the condition number of the matrix rises
so does the number of iterations required for the solver to converge.
For FEM the number of iterations increases from 69 for the best mesh
to 831 for the most deformed mesh. The numbers of iterations for AES-FEM
and GFD remain almost constant, increasing from 56 to 59 and from
56 to 60, respectively. See Figure~\ref{fig:3D_poor_quality_meshes}
for a comparison of the condition numbers and iteration counts of
the solvers.

For each of the four methods, the errors were nearly constant over
the series of meshes. For FEM, the $L_{2}$ error on the 1st mesh
was $4.36\times10^{-4}$ and on the 6th mesh, the error was $4.37\times10^{-4}$.
For AES-FEM 1, the $L_{2}$ error on all the meshes was $2.92\times10^{-4}$.
For AES-FEM 2, the $L_{2}$ error on all the meshes was $1.30\times10^{-4}$.
For GFD, the $L_{2}$ error on the 1st mesh was $3.43\times10^{-5}$
and on the 6th mesh, the error was $3.40\times10^{-5}$.

\begin{figure}
\begin{minipage}[t]{0.45\textwidth}%
\begin{center}
\includegraphics[width=1\textwidth]{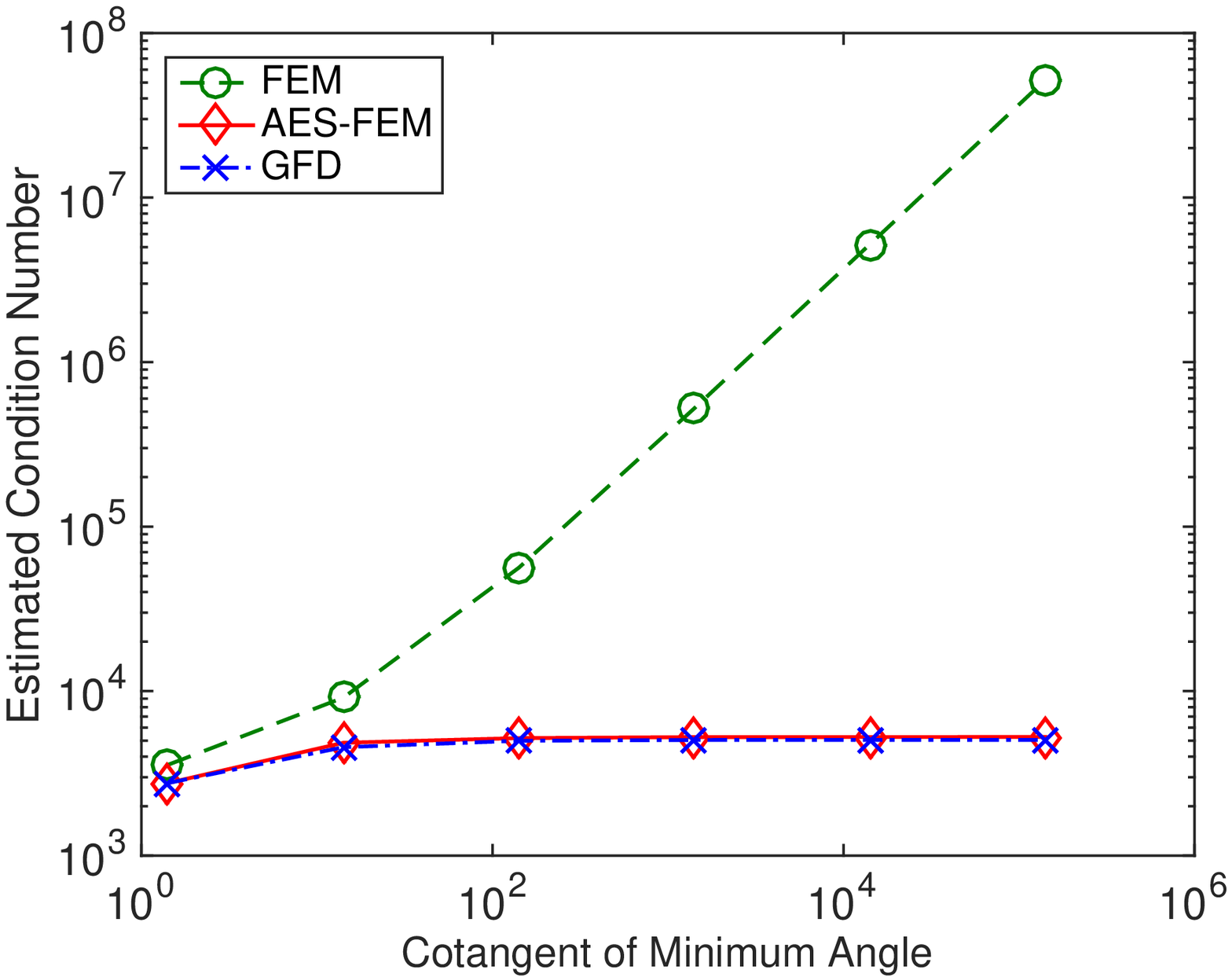}
\par\end{center}%
\end{minipage}\hfill{} %
\begin{minipage}[t]{0.45\textwidth}%
\begin{center}
\includegraphics[width=1\textwidth]{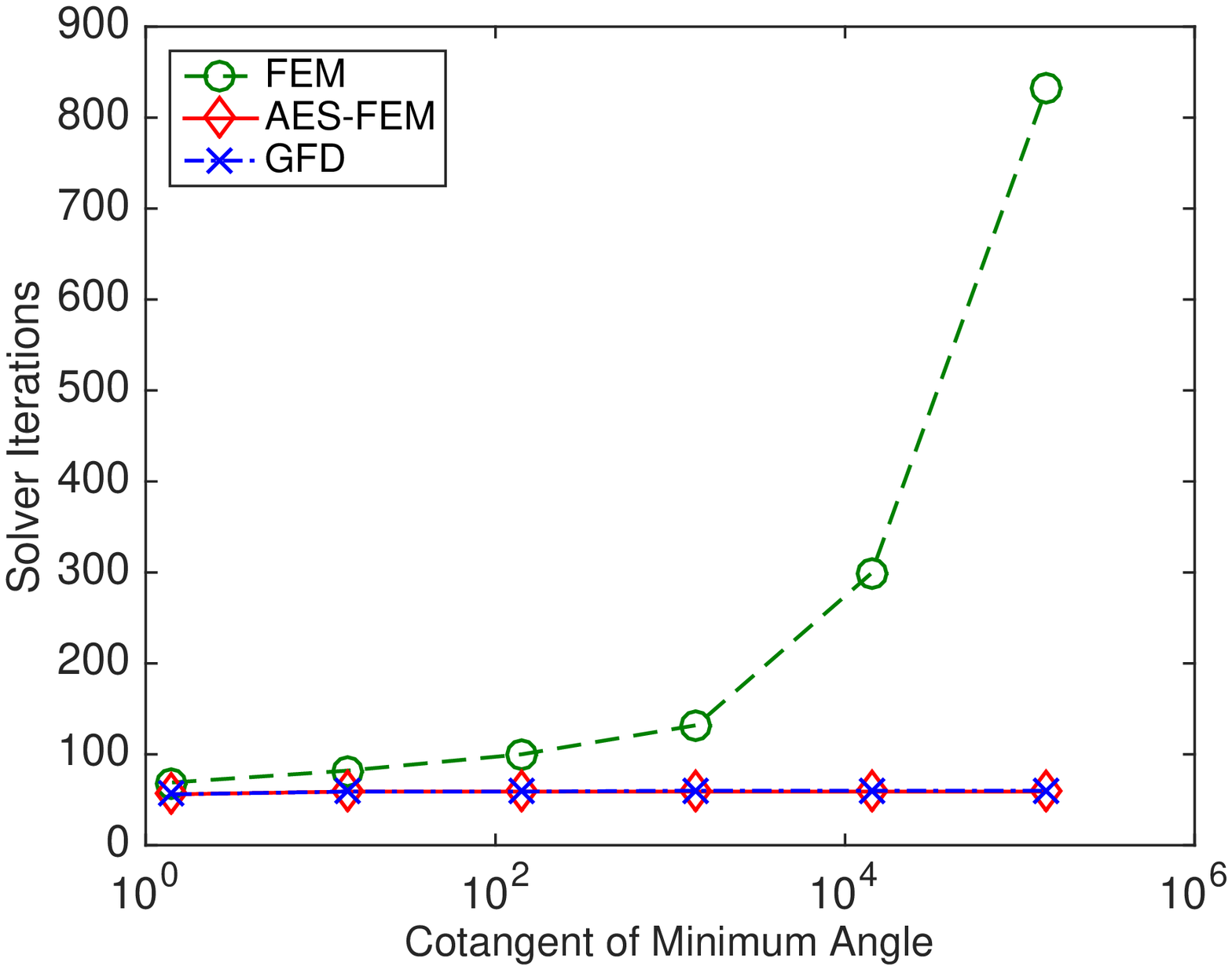}
\par\end{center}%
\end{minipage}

\raggedright{}\caption{The condition numbers of the stiffness matrices for FEM and AES-FEM
and the coefficient matrix for GFD (left) and the numbers of solver
iterations (right). Solvers used are preconditioned CG for FEM and
preconditioned GMRES for AES-FEM and GFD.\label{fig:3D_poor_quality_meshes}}
\end{figure}

\subsubsection{Efficiency}

We compare the runtimes of the four methods for solving the convection-diffusion
equation with the polynomial analytic solution $u_{1}$ on the most
refined mesh of series 2. As in 2D, the total time is decomposed into
4 subcategories: \textit{Initialization,} \textit{Assembly}, \textit{Preconditioner},
and \textit{Solver}. The preconditioner used is incomplete LU with
a drop tolerance of $10^{-1}$. GMRES with a tolerance of $10^{-8}$
is used as the solver. AES-FEM 1 and AES-FEM 2 each require 142 iterations
to converge, FEM requires 128 iterations, and GFD requires 138 iterations.
The majority of the time is spent assembling the matrix and solving
the system. See Figure \ref{fig:time_3D_CD} for the comparison. 

As in 2D, FEM is the most efficient method on a given mesh. Overall,
the total runtime of FEM is approximately 1.7 times faster than AES-FEM
1, 1.9 times faster than AES-FEM 2, and 1.8 times faster than GFD.
The assembly of FEM is approximately 5.6 times faster than AES-FEM
1, 6.5 times faster than AES-FEM 2, and 1.2 times faster than GFD.

In 3D, the difference of assembling the load vector using FEM basis
functions (AES-FEM 1) versus using GLP basis functions (AES-FEM 2)
is more pronounced than in 2D. The assembly in AES-FEM 1 is 8.3 seconds
shorter than that of AES-FEM 2. This means the assembly of AES-FEM
1 uses 14.4\% less time than that of AES-FEM 1 and the total runtime
is 8.7\% shorter. 

\begin{figure}
\begin{centering}
\includegraphics[scale=0.6]{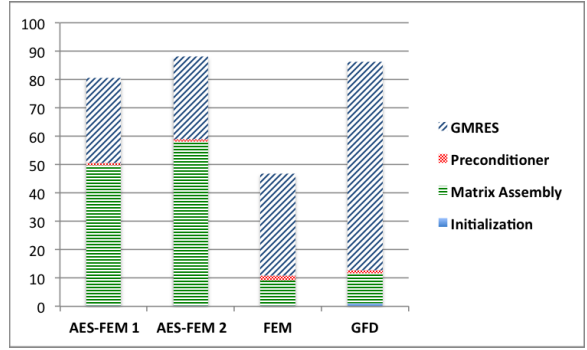}
\par\end{centering}

\caption{Runtimes for a 3D convection-diffusion equation on the most refined
mesh in series 2.\label{fig:time_3D_CD}}
\end{figure}

However, similar to 2D, AES-FEM is competitive with, and most of time
more efficient than, the classical FEM with linear basis functions
in terms of error versus runtime. Figure~\ref{fig:time_vs_error_3d}
shows the $L_{\infty}$ norm errors versus runtimes for the four methods
on mesh series 2 for the Poisson equation and the convection-diffusion
equation for $u_{2}$. For the Poisson equation, GFD is more efficient
on coarser meshes and AES-FEM 2 is more efficient for finer meshes.
For the convection-diffusion equation, GFD is more efficient on smaller
meshes and AES-FEM 2 is more efficient for finer meshes. AES-FEM 1
is also more efficient than FEM.

\begin{figure}
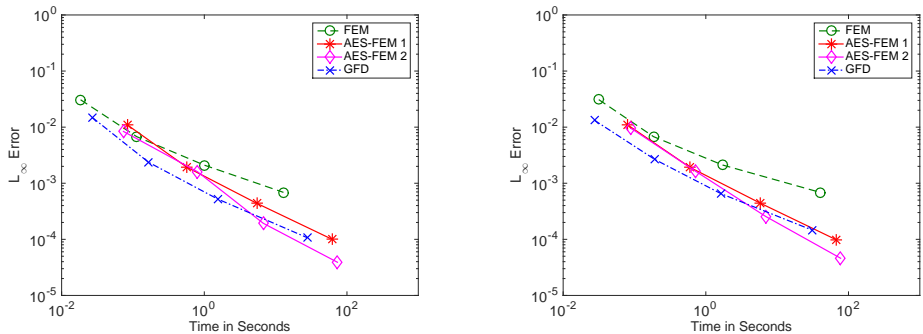

\begin{minipage}[t]{0.45\textwidth}%
\begin{center}
\includegraphics[width=1\textwidth]{error_vs_time_pois_mesh2_u2_inf_3d}
\par\end{center}%
\end{minipage}\hfill{} %
\begin{minipage}[t]{0.45\textwidth}%
\begin{center}
\includegraphics[width=1\textwidth]{error_vs_time_cd_u2_mesh2_inf_3d}
\par\end{center}%
\end{minipage}

\raggedright{}\caption{$L_{\infty}$ norm errors versus runtimes for a 3D Poisson equation
(left) and convection-diffusion equation (right) on mesh series 2.
Lower is better.\label{fig:time_vs_error_3d} }
\end{figure}

\section{Conclusions and Future Work\label{sec:Conclusions-and-Future}}

In this paper, we proposed the adaptive extended stencil finite element
method, which uses generalized Lagrange polynomial basis functions
constructed from weighted least squares approximations. The method
preserves the theoretical framework of the classical FEM and the simplicity
in imposing essential boundary conditions and integrating the stiffness
matrix. We presented the formulation of AES-FEM, showed that the method
is consistent, and discussed both the local and global stability of
the method. We described the implementation, including the mesh data
structure and the numerical algorithms. We compared the accuracy of
AES-FEM against the classical FEM with linear basis functions and
the quadratic generalized finite difference method for the Poisson
and convection-diffusion equations in both 2D and 3D. We showed improved
accuracy and stability of AES-FEM over FEM, and demonstrated that
the condition number of AES-FEM, and hence the convergence rate of
iterative solvers, are independent of the element quality of the mesh.
Our experiments also showed that AES-FEM is more efficient than the
classical FEM in terms of error versus runtime, while having virtually
the same sparsity patterns of the stiffness matrices.

As a general method, AES-FEM can use generalized Lagrange polynomial
basis functions of arbitrary degrees. We only focused on quadratic
basis functions in this paper. In future work, we will report higher-order
AES-FEM with cubic and higher-degree basis functions. The present
implementation of AES-FEM uses the standard hat functions as the weight
functions, which may lead to large errors when applied to tangled
meshes with inverted elements. We will report the resolution of tangled
meshes in a future publication. Finally, while AES-FEM is efficient
in terms of error versus runtime, it is much slower than the classical
FEM on a given mesh due to the slower computation of the basis functions
and the nonsymmetry of the stiffness matrix. The efficiency can be
improved substantially by leveraging parallelism and efficient multigrid
solvers, which we will report in the future.

\section*{Acknowledgements}

This work was supported by DoD-ARO under contract \#W911NF0910306.
The third author is also supported by a subcontract to Stony Brook
University from Argonne National Laboratory under Contract DE-AC02-06CH11357
for the SciDAC program funded by the Office of Science, Advanced Scientific
Computing Research of the U.S. Department of Energy.

\bibliographystyle{../wileyj}
\bibliography{../refs/aes-fem}

\end{document}